\documentclass[a4paper, 11pt]{article}
\usepackage{amsmath,amssymb,esint,amscd,xspace,fancyhdr,color,authblk,srcltx,fontenc,bbm}
\usepackage{hyperref}
\usepackage{cite}

\makeatletter
\newcommand{\leqnomode}{\tagsleft@true}
\newcommand{\reqnomode}{\tagsleft@false}
\makeatother

\newtheorem{theorem}{Theorem}[section]

\newtheorem{definition}[theorem]{Definition}
\newtheorem{lemma}[theorem]{Lemma}

\newtheorem{remark}[theorem]{Remark}
\newenvironment{proof}[1][Proof]{\textbf{#1.} }{\hfill\rule{0.5em}{0.5em}}
{\catcode`\@=11\global\let\AddToReset=\@addtoreset
\AddToReset{equation}{section}

\AddToReset{theorem}{section}

\title{Lorentz improving estimates for the $p$-Laplace equations with mixed data}
\author{Thanh-Nhan Nguyen\thanks{Department of Mathematics, Ho Chi Minh City University of Education, Ho Chi Minh City, Vietnam; \texttt{nhannt@hcmue.edu.vn}}, Minh-Phuong Tran\thanks{Applied Analysis Research Group, Faculty of Mathematics and Statistics, Ton Duc Thang University, Ho Chi Minh City, Vietnam; \texttt{tranminhphuong@tdtu.edu.vn}} \footnote{Corresponding author.}}

\date{\today}

\begin{document}
 
\maketitle
\begin{abstract}
The aim of this paper is to develop the regularity theory for a weak solution to a class of quasilinear nonhomogeneous elliptic equations, whose prototype is the following mixed Dirichlet $p$-Laplace equation of type
\begin{align*}
\begin{cases}
\mathrm{div}(|\nabla u|^{p-2}\nabla u) &= f+ \ \mathrm{div}(|\mathbf{F}|^{p-2}\mathbf{F}) \qquad  \text{in} \ \ \Omega, \\
\hspace{1.2cm} u &=\ g \hspace{3.1cm} \text{on} \ \ \partial \Omega,
\end{cases}
\end{align*}
in Lorentz space, with given data $\mathbf{F} \in L^p(\Omega;\mathbb{R}^n)$, $f \in L^{\frac{p}{p-1}}(\Omega)$, $g \in W^{1,p}(\Omega)$ for $p>1$ and $\Omega \subset \mathbb{R}^n$ ($n \ge 2$) satisfying a Reifenberg flat domain condition or a $p$-capacity uniform thickness condition, which are considered in several recent papers. To better specify our result, the proofs of regularity estimates involve fractional maximal operators and valid for a more general class of quasilinear nonhomogeneous elliptic equations with mixed data. This paper not only deals with the Lorentz estimates for a class of more general problems with mixed data but also improves the good-$\lambda$ approach technique proposed in our preceding works~\cite{MPT2018,PNCCM,PNJDE,PNCRM}, to achieve the global Lorentz regularity estimates for gradient of weak solutions in terms of fractional maximal operators.

\medskip

\medskip

\medskip

\noindent 

\medskip


\end{abstract}   
                  
\section{Introduction}\label{sec:intro}

The main outcome of this paper is to establish some global regularity estimates for solutions of the $p$-Laplacian equations with mixed data of the type
\begin{align}
\label{eq:pLap}
\mathrm{div}(|\nabla u|^{p-2}\nabla u) &= f+ \ \mathrm{div}(|\mathbf{F}|^{p-2}\mathbf{F}) ,
\end{align}
with general Dirichlet boundary condition $g \in W^{1,p}(\Omega)$, where $\mathbf{F} \in L^p(\Omega;\mathbb{R}^n)$ and $f \in L^{\frac{p}{p-1}}(\Omega)$ for $p>1$. Specifically, the gradient estimates of the weak solutions are obtained in Lorentz spaces $L^{q,s}(\Omega)$ in terms of \emph{fractional maximal operators}. In fact, the results shown here also hold for a larger class of  equations that are more general than one in~\eqref{eq:pLap}. We in detail consider the quasilinear elliptic problems with mixed data (under nonhomogeneous Dirichlet boundary condition) of the type
\begin{align}\label{eq:main}
\begin{cases}
\mathrm{div}(\mathbb{A}(x,\nabla u)) &= f + \ \mathrm{div}(\mathbb{B}(x,\mathbf{F})) \quad \text{in} \ \ \Omega, \\
\hspace{1.2cm} u &=\ g \hspace{3cm} \text{on} \ \ \partial \Omega.
\end{cases}
\end{align}
In other words, the $p$-Laplace equation \eqref{eq:pLap} is the prototype of a class of quasilinear elliptic equations in \eqref{eq:main}. 
Here, in our study, $\Omega \subset \mathbb{R}^n$ is an open bounded domain, $(n \ge 2)$ and given data $\mathbf{F} \in L^p(\Omega;\mathbb{R}^n)$, $f \in L^{\frac{p}{p-1}}(\Omega)$, and the problem is set up with nonhomogeneous boundary data $g \in W^{1,p}(\Omega)$ for $p>1$. Moreover, the nonlinear operators $\mathbb{A}, \mathbb{B}: \Omega \times  \mathbb{R}^n \rightarrow \mathbb{R}^n$ are Carath\'eodory vector valued functions (they are measurable on $\Omega$ for every $z$ in $\mathbb{R}^n$ with respect to $x$, and continuous on $\mathbb{R}^n$ for almost every $x$ in $\Omega$ with respect to $z$) which satisfy the following growth and monotonicity conditions: there exist $\Lambda_1$, $\Lambda_2>0$ such that two following inequalities 
\begin{align}\label{cond:A1}
\left|\mathbb{A}(x,z) \right| + \left|\mathbb{B}(x,z)\right| &\le \Lambda_1 |z|^{p-1};
\end{align}
\begin{align}\label{cond:A2}
\langle \mathbb{A}(x,z_1)-\mathbb{A}(x,z_2), z_1 - z_2 \rangle &\ge \Lambda_2 {\Phi}(z_1,z_2),
\end{align}
hold for almost every $x$ in $\Omega$ and every $z_1$, $z_2$ in $ \mathbb{R}^n \setminus \{0\}$, where the function ${\Phi}$ in~\eqref{cond:A2} is defined by
\begin{align}\label{def:G}
{\Phi}(z_1,z_2) =  \left( |z_1|^2 + |z_2|^2 \right)^{\frac{p-2}{2}}|z_1 - z_2|^2, \qquad z_1, \, z_2 \in \mathbb{R}^n \setminus \{0\}.
\end{align}

1.1. \textbf{Relation to prior works.} Before diving into our results in this paper, we start by briefly describing some related results pertaining to regularity theory also the classical Calder\'on-Zygmund theory for the nonlinear elliptic equations in recent years. As a starting point to that theory,  the study of the regularity results for classical  $p$-Laplacian equations, which is interesting in itself. And later, regularity theory for problems in the more general setting of $p$-Laplacian type equations have been further studied.  We refer to the classical results by the series of papers~\cite{Ural1968, LadyUral1968, Uhlenbeck1977, Evans1982, Tolksdorf1984, Lewis1983, DiBenedetto1983, Lieberman1984, Lieberman1986, Lieberman1988, Iwaniec}, for reliable results of these pioneering works. In these works, attention has been driven to the interior $C^{1,\alpha}$ regularity for the weak solutions to the class of $p$-Laplacian equations, and later for homogeneous quasilinear elliptic equations of the type $-\mathrm{div}(\mathbb{A}(x,u,\nabla u))=0$, or the divergence form modeled by the $p$-Laplacian equation $\mathrm{div}(|\nabla u|^{p-2}\nabla u)=\mathrm{div}(|\mathbf{F}|^{p-2}\mathbf{F})$ have been greatly developed through the years. 

Let us discuss on some regularity results of solutions to quasilinear elliptic equations formulated in the more general form, that attract interest to the researchers in years, as following
\begin{align}
\label{eq:A-RHS}
\mathrm{div}(\mathbb{A}(x,\nabla u)) = \texttt{RHS},
\end{align}
where the \texttt{RHS} (right hand side) may be given in divergence or not in divergence form, or even a measure data. For instance, when the \texttt{RHS} is given in $p$-Laplacian form: $\mathrm{div}(\mathbb{A}(x,\nabla u)) = \ \mathrm{div}(|\mathbf{F}|^{p-2}\mathbf{F})$, with zero Dirichlet boundary data, there are good results for interior and global regularity results established by seminal works of S.-S. Byun \textit{et al.} in~\cite{BW1_1, SSB2, SSB3, SSB4, SSB1}, T. Mengesha \textit{et al.} in~\cite{MP11, MP12} in the setting of classical $L^p$ spaces and its weighted version, Sobolev spaces $W^{\alpha,p}$, respectively. It is worth noting that in these works, the problem is set under assumptions on Reifenberg domain $\Omega$, together with standard ellipticity condition of $\mathbb{A}$ and small BMO oscillation in $x$. Besides, there are other approaches showing results  of this equation under various assumptions on the data, such as~\cite{Phuc2, PNJDE} for weaker hypothesis on domain $\Omega$ (its complement satisfies $p$-capacity uniformly thickness condition), obtained in Lorentz spaces. On the other hand, a plenty of interesting regularity results have been obtained for weak solutions to~\eqref{eq:A-RHS}, where \texttt{RHS} is given in divergence form $\mathrm{div}(\mathbb{A}(x,\nabla u)) = \mathrm{div} (\mathbf{F})$. For instance, the interior $W^{1,q}$ estimates of solutions was done by T. Nguyen \textit{et al.} in a remarkable paper~\cite{Truyen2016}; and later these results were recently generalized in\cite{FT2018, Truyen2018} for the global gradient estimates in weighted Morrey spaces. Moreover, in some intensive works, many authors dealt with regularity results in several spaces, which require different assumptions on the domain $\Omega$ and the nonlinearity $\mathbb{A}$ and the given boundary data (see, e.g.,~\cite{BW1, GW, Tuoc2018, Phuc2015, BCDKS, CM2014, CoMi2016, KZ, MPTNsub}) to this class of elliptic equations. 


The problem~\eqref{eq:A-RHS} where the data  \texttt{RHS} mixed between divergence and nondivergence forms $\mathrm{div}(|\mathbf{F}|^{p-2}\mathbf{F}) + f$ (with homogeneous Dirichlet boundary condition) was recently considered by M. Lee \textit{et al.} in~\cite{Lee2019}, that was motivated by some previous works by V. B\"ogelein \textit{et al.} in~\cite{BDM2011, Boge2014} resulting the global Calder\'on-Zygmund theorem in the context of inhomogeneous parabolic systems of $p$-Laplacian type. 

Inspired by these recent works above-mentioned and mathematical techniques developed for nonlinear elliptic equations, our investigation here is to establish the global regularity estimates for weak solutions to the nonhomogeneous problem~\eqref{eq:main} in Lorentz spaces, in which the use of  ``good-$\lambda$ type bounds'' technique plays a crucial role. By far, the so-called terminology ``good-$\lambda$'' is one of the most effective techniques devoted to study of regularity estimates (in Euclidean setting, see~\cite{MW}).  This approach has been studied and discussed by a great number of research papers, such as~\cite{55QH4,MPT2018,PNCCM,PNJDE} to which we refer the interested readers. In addition, our preceding works in~\cite{PNCRM,PNnonuniform} presented the generalized good-$\lambda$ technique and its applications to regularity estimates for the uniform/non-uniformly elliptic equations. Based upon the studies, it is worth emphasizing that in this paper, we additionally improve the proof techniques for ``good-$\lambda$ type bounds'' method to obtain better regularity results for interior as well as up to the boundary of domain. To be more concrete, this study elaborates another route to prove regularity estimates in Lorentz spaces. Completely avoiding the use of cut-off fractional maximal functions as in previous studies~\cite{PNJDE,MPTNsub,PNnonuniform}, we believe that that the proposed technique described in this paper can be beneficially used to obtain gradient estimates of solutions for problems~\eqref{eq:main} in general, that are preserved under the fractional  maximal operators. Fractional maximal operators were first employed by F. Duzaar {\it et al.} in \cite{Duzamin2, 55DuzaMing,KM2012, KM2014} and then are widely used in several research papers to obtain gradient estimates for solutions to elliptic/parabolic equations (see~\cite{Mi1, KM2013,Min2003, Min2007} for more interesting papers, and even~\cite{Mi2019} for an  article of history, recently written by G. Mingione {\it et al.}). This work not only extends and improves the results in our previous work~\cite{PNJDE} (for problems with divergence form data only) to the problem with mixed data but also develops results in~\cite{Lee2019} (for Calder\'on-Zygmund estimates in the setting of Lebesgue spaces) to the Lorentz spaces. 

To better specify our results, through the paper there are two types of domain hypothesis formulated separately to compare results in the framework of Lorentz spaces: \emph{domains with uniformly $p$-capacity thick complement} versus \emph{Reifenberg flat domains}, in which their central concepts and definitions will be introduced underneath. 

Before stating our main results, let us first introduce some notation and preliminary definitions that will be used in the rest of the paper.\\

1.2. \textbf{Notation and Definitions.}
\begin{itemize}
\item We use the symbol $C$ to denote positive constants that not necessarily the same at each occurrence depending only on dimension and some constants appearing in the theorems. The dependence of $C$ on some prescribed parameters will be emphasized between parentheses. Further, all constants starting by $C$, such as $C, C_i$ for example, are assumed to be larger than or equal to one.
\item The given domain $\Omega \subset \mathbb{R}^n$, $n \ge 2$ is assumed to be an open bounded domain. 
\item We let $\mathcal{L}^n(E)$ stand for the Lebesgue measure of a measurable set $E$ in $\mathbb{R}^n$. In addition, we will write $\mathrm{diam}(\Omega)$ for the diameter of $\Omega$, defined as:
\begin{align*}
\mathrm{diam}(\Omega) = \sup_{\xi_1, \xi_2 \in \Omega} |\xi_1-\xi_2|.
\end{align*}
\item We shall denote the integral average of a function $h \in L^1_{\mathrm{loc}}(\mathbb{R}^n)$ over the measurable subset  $E$ of $\mathbb{R}^n$ as
\begin{align*}
\fint_E{h(x)dx} = \frac{1}{\mathcal{L}^n(E)}\int_E{h(x)dx}.
\end{align*}
\item The open $n$-dimensional Euclidean ball in $\mathbb{R}^n$ of radius $\varrho>0$ and center $\xi$ will be denoted by $B_{\varrho}(\xi)$, that is the set $\{z \in \mathbb{R}^n: |z-\xi|<\varrho\}$. We also denote $\Omega_{\varrho}(\xi) = B_{\varrho}(\xi) \cap \Omega$ which is considered as the ``surface ball'' when the center $\xi$ lies on $\partial\Omega$.
\end{itemize}
\begin{definition}[Weak solution]\label{def:weak_sol}
A function $u \in W^{1,p}(\Omega)$ is a weak (distributional) solution to~\eqref{eq:main} under the assumptions~\eqref{cond:A1} and~\eqref{cond:A2} if 
\begin{align*}
\int_\Omega{ \langle\mathbb{A}(x,\nabla u), \nabla\varphi \rangle dx} = \int_\Omega{\langle \mathbb{B}(x,\mathbf{F}), \nabla\varphi \rangle dx} - \int_\Omega{f\varphi dx} ,
\end{align*}
holds for all $\varphi \in W_0^{1,p}(\Omega)$.
\end{definition}

\begin{definition}[The $p$-capacity]
Let $\Omega \subset \mathbb{R}^n$, $n \ge 2$ be a bounded domain. The $p$-capacity of a compact subset $\mathcal{K} \subset \Omega$ with respect to $\Omega$ is defined to be
\begin{align*}
\mathrm{cap}_p(\mathcal{K},\Omega) = \inf \left\{ \int_\Omega{|\nabla \psi|^p dx}: \ \psi \in C_c^\infty(\Omega), \ \psi \ge 1 \ \text{on} \ \mathcal{K} \right\}.
\end{align*}
If $\mathcal{O} \subseteq \Omega$ is an open set, then
\begin{align*}
\mathrm{cap}_p(\mathcal{O},\Omega) = \sup \left\{ \mathrm{cap}_p(\mathcal{K},\Omega): \ \mathcal{K} \subseteq \mathcal{O}, \  \mathcal{K} \ \text{compact} \right\},
\end{align*}
meanwhile, the $p$-capacity of arbitrary set $\mathcal{A} \subseteq \Omega$ is defined by
\begin{align*}
\mathrm{cap}_p(\mathcal{A},\Omega) = \inf \left\{ \mathrm{cap}_p(\mathcal{O},\Omega): \ \mathcal{O} \subseteq \mathcal{A}, \ \mathcal{O} \ \text{open} \right\}.
\end{align*}
\end{definition}
\begin{definition}[Domains with uniformly $p$-capacity thick complement]\label{def:pcapa}
Let $\Omega$ be an arbitrary bounded domain of $\mathbb{R}^n$ ($n \ge 2$). Then, domain $\mathbb{R}^n \setminus \Omega$ is said to satisfy the $p$-capacity uniform thickness condition if there exist two constants $c_0,r_0>0$ such that
\begin{align}\tag{$\mathcal{HP}$}
\label{hyp:P}
\mathrm{cap}_p((\mathbb{R}^n \setminus \Omega) \cap \overline{B}_{\varrho}(\xi), B_{2\varrho}(\xi)) \ge c_0 \mathrm{cap}_p(\overline{B}_{\varrho}(\xi),B_{2\varrho}(\xi)),
\end{align}
for every $\xi \in \mathbb{R}^n \setminus \Omega$ and $0<\varrho \le r_0$. 
\end{definition}
\begin{remark}
Every nonempty $\mathbb{R}^n \setminus \Omega$ is uniformly $p$-thick for $p > n$, and therefore the uniform capacity density condition is nontrivial only when $p \le n$. Moreover, if $\mathbb{R}^n \setminus \Omega$ is uniformly $p$-thick, then it is uniformly $q$-thick for all $q \ge p$.
\end{remark}
\begin{remark}
We note that the $p$-capacity uniform thickness implies that all points on $\partial\Omega$ is regular. According to~\cite{Kilp, HKM2006}, a sufficient condition for $\xi \in \partial\Omega$ being regular is
\begin{align*}
\int_0^1{\left(\frac{\mathrm{cap}_p((\mathbb{R}^n \setminus \Omega) \cap \overline{B}_{\varrho}(\xi), B_{2\varrho}(\xi))}{\mathrm{cap}_p(\overline{B}_{\varrho}(\xi), B_{2\varrho}(\xi))} \right)^{\frac{1}{p-1}} \frac{d\varrho}{\varrho}} = \infty,
\end{align*}
for the $p$-Laplace equation, where the thickness of complement of $\Omega$ near boundary $\partial \Omega$ can be measured by capacity densities. This condition is called the Wiener criterion, first introduced by N. Wiener in~\cite{Wiener}, that is important in regularity of boundary points.
\end{remark}
\begin{remark}
Such assumption~\eqref{hyp:P} is very mild and essential for higher integrability results. Domains whose complement satisfy $p$-capacity uniform thickness include domains with Lipschitz continuous boundaries or satisfy a uniform exterior corkscrew condition. 
\end{remark}

\begin{definition}[$(\delta,r_0)$-Reifenberg flat domain]\label{def:Reifenberg}
Let $\Omega \subset \mathbb{R}^n$, $n \ge 2$ be a bounded domain and let $\delta \in (0,1)$ and $r_0>0$. Then, we say that $\Omega$ is a $(\delta,r_0)$-Reifenberg flat domain if for each $\xi_0 \in \partial \Omega$ and each $\varrho \in (0,r_0]$, one can find a coordinate system $\{y_1,y_2,...,y_n\}$ with origin at $\xi_0$ such that
\begin{align}\tag{$\mathcal{HR}$}
\label{hyp:R}
B_{\varrho}(\xi_0) \cap \{y: \ y_n > \delta \varrho\} \subset B_{\varrho}(\xi_0) \cap \Omega \subset B_{\varrho}(\xi_0) \cap \{y: \ y_n > -\delta \varrho\},
\end{align}
where for the sake of convenience, the set $\{y = (y_1, y_2, ..., y_n): \ y_n > c\}$ is denoted by $\{y: \ y_n > c\}$.
\end{definition}
\begin{remark}
For a suitable regularity parameter $\delta$, any Lipschitz domain with small Lipschitz constant is a Reifenberg flat domain (value of $\delta$ depends on the Lipschitz constant) and even some domains with fractal boundaries, see~\cite{Toro1997}. 
\end{remark}

\begin{remark}\label{rem:2types}
There are two types of hypothesis on domain $\Omega$ separately considered in this paper: domains under Definition~\ref{def:pcapa} and~\ref{def:Reifenberg}, respectively. It is worth noting that hypothesis~\eqref{hyp:P} is weaker than one in~\eqref{hyp:R}. Geometrically, the Reifenberg domains are domains which are flat in the sense that their boundaries are well-approximated by planes/hyperplanes. This class of domains include all $C^1$ domains, domains with small Lipschitz constants, domains with fractal boundaries, etc. Reifenberg flatness domains may be very rough but assumed to be sufficiently flat, we can say that this concept is a ``minimal regularity requirement" for the boundary $\partial\Omega$ ensuring the main results of the geometric analysis continue to hold true in $\Omega$. Elliptic and parabolic equations in Reifenberg flat domains were exploited by several authors in~\cite{BW2,SSB4, MT2010, LM1} and many references therein.
\end{remark}

\begin{definition}[$(\delta,r_0)$-BMO condition]\label{def:BMOcond}
The nonlinearity $\mathbb{A}$ is said to satisfy a $(\delta,r_0)$-BMO condition with exponent $t>0$ if the following condition holds
\begin{align}\label{cond:BMO}
[\mathbb{A}]_t^{r_0} = \sup_{y \in \mathbb{R}^n, \ 0<\varrho\le r_0} \left(\fint_{B_{\varrho}(y)} \left(\sup_{z \in \mathbb{R}^n \setminus \{0\}} \frac{|\mathbb{A}(x,z) - \overline{\mathbb{A}}_{B_{\varrho}(y)}(z)|}{|z|^{p-1}}\right)^t dx\right)^{\frac{1}{t}} \le \delta,
\end{align}
where $\overline{\mathbb{A}}_{B_{\varrho}(y)}(z)$ denotes the average of $\mathbb{A}(\cdot,z)$ over the ball $B_{\varrho}(y)$.
\end{definition}

\begin{definition}[Lorentz spaces]\label{def:Lorentz}
Lorentz space $L^{q,s}(\Omega)$ for $0<q<\infty$ and $0 <s< \infty$ is defined by that for all Lebesgue measurable function $h$ on $\Omega$, there holds
\begin{align*}
\|h\|_{L^{q,s}(\Omega)} := \left[ q \int_0^\infty{ \tau^s\mathcal{L}^n \left( \{\zeta \in \Omega: |h(\zeta)|>\tau\} \right)^{\frac{s}{q}} \frac{d\tau}{\tau}} \right]^{\frac{1}{s}} < \infty.
\end{align*}
Otherwise, when $s=\infty$, the space $L^{q,\infty}(\Omega)$ is  Marcinkiewicz spaces with the following quasi-norm
\begin{align*}
\|h\|_{L^{q,\infty}(\Omega)} := \sup_{\tau>0}{\tau \left[\mathcal{L}^n\left(\{\zeta \in \Omega:|h(\zeta)|>\tau\}\right)\right]^{\frac{1}{q}}}.
\end{align*}
\end{definition}
\begin{remark}
When $s=t$ the Lorentz space $L^{s,s}(\Omega)$ coincides with usual Lebesgue space $L^s(\Omega)$. In particular, we know that for some $0<r \le s \le t \le \infty$, there holds $L^t(\Omega)  \subset L^{s,r}(\Omega) \subset  L^{s}(\Omega) \subset L^{s,t}(\Omega) \subset L^r(\Omega)$. One also finds in~\cite{Gra97} for some further properties regarding these spaces.
\end{remark}
  
\begin{definition}[Fractional maximal function]\label{def:Malpha}
For $\alpha \in [0, n]$, the fractional maximal operator $\mathbf{M}_{\alpha}$ of a function $h \in L^1_{\mathrm{loc}}(\mathbb{R}^n)$ is given by:
\begin{align} \nonumber 
\mathbf{M}_\alpha h(x) = \sup_{\varrho>0}{{\varrho}^{\alpha} \fint_{B_{\varrho}(x)}{|h(y)|dy}}, \quad x \in \mathbb{R}^n.
\end{align}
In a special case when $\alpha=0$, it coincides with the well-known Hardy-Littlewood maximal function $\mathbf{M}$ defined by:
\begin{align}\nonumber 
\mathbf{M}h(x) = \sup_{\varrho >0}{\fint_{B_{\varrho}(x)}|h(y)|dy}, \quad x \in \mathbb{R}^n,
\end{align}
for a given locally integrable function $h$ in $\mathbb{R}^n$.
\end{definition}
In this parer, the boundedness property of fractional maximal operators play a crucial role for gradient estimates of the distribution solution to our problem.  The following lemma will recall this useful property of fractional maximal operators. A detail proof of this lemma can be found in our previous work~\cite{PNnonuniform}. Moreover, we remark that when $\alpha = 0$ this property is exactly the boundedness property of Hardy-Littlewood maximal function in~\cite{55Gra}.
\begin{lemma}[Boundedness of fractional maximal function $\mathbf{M}_{\alpha}$, see~\cite{PNnonuniform}]\label{lem:M_alpha}
The operator $\mathbf{M}_{\alpha}$ is bounded from $L^s(\mathbb{R}^n)$ to $L^{\frac{ns}{n-\alpha s},\infty}(\mathbb{R}^n)$, for $s \ge 1$ and $\alpha \in \left[0,\frac{n}{s}\right)$. This means there is a positive constant $C=C(n,s,\alpha)$ such that 
\begin{align*}
\mathcal{L}^n\left(\left\{\zeta \in \mathbb{R}^n: \ \mathbf{M}_{\alpha}h(\zeta)> \tau\right\}\right) \le C \left(\frac{1}{\tau^{s}}\int_{\mathbb{R}^n}|h(y)|^s dy\right)^{\frac{n}{n-\alpha s}},
\end{align*}
for all $h \in L^s(\mathbb{R}^n)$ and $\tau >0$.
\end{lemma}

1.3. \textbf{Statement of the Main Results.} From now on, we always consider $\Omega$ as an open bounded domain of $\mathbb{R}^n$ and $u \in W^{1,p}(\Omega)$ as a weak solution to equation~\eqref{eq:main} with given data $\mathbf{F} \in L^p(\Omega)$, $f \in L^{\frac{p}{p-1}}(\Omega)$ and Dirichlet boundary condition $g \in W^{1,p}(\Omega)$ for $p \in (1,n]$. For simplicity of notation, let us denote by $|\mathcal{E}|^p = |\mathbf{F}|^p + |f|^{\frac{p}{p-1}} + |\nabla g|^p$ in the whole paper. Moreover, the operator $\mathbb{A}$ is assumed satisfying assumptions~\eqref{cond:A1} and~\eqref{cond:A2} with suitable constants $\Lambda_1, \Lambda_2$. Main results of the present paper can be stated as follows. 
\begin{theorem}[Good-$\lambda$ theorem under assumption~\eqref{hyp:P}]
\label{theo:good-lambda} 
Assume that domain $\Omega \subset \mathbb{R}^n$ satisfies the $p$-capacity thick complement condition~\eqref{hyp:P} with two constants $c_0$, $r_0 \in \mathbb{R}^+$. Then one can find two constants $\gamma = \gamma(n,p,\Lambda_1,\Lambda_2) >1$ and $\theta  =\theta (p) \ge 2$ such that for every $0 \le \alpha < \frac{n}{\gamma}$ the following estimate
\begin{multline}\label{eq:good-lambda}
\mathcal{L}^n\left(\left\{\zeta \in \Omega: \ {\mathbf{M}}_{\alpha}(|\nabla u|^p)(\zeta) >\varepsilon^{\frac{\alpha}{n}-\frac{1}{\gamma}} \lambda, \ {\mathbf{M}}_{\alpha}(|\mathcal{E}|^p)(\zeta) \le \varepsilon^{\theta \left(1-\frac{1}{\gamma}\right)} \lambda \right\} \right) \\
 \hspace{4cm} \le C \varepsilon \mathcal{L}^n\left(\left\{\zeta \in \Omega: \ {\mathbf{M}}_{\alpha} (|\nabla u|^p)(\zeta) > \lambda\right\}\right),
\end{multline}
holds for all $\lambda \in \mathbb{R}^+$ and $\varepsilon \in (0,\varepsilon_0)$, where $\varepsilon_0 = \varepsilon_0(n,p,\alpha,\Lambda_1,\Lambda_2) \in (0,1)$ and  $C = C(n, p, \Lambda_1, \Lambda_2,  \alpha, c_0, r_0, \mathrm{diam}(\Omega))>0$.
\end{theorem}
\begin{theorem}[Global Lorentz regularity estimate under assumption~\eqref{hyp:P}]
\label{theo:main}
Assume that domain $\Omega \subset \mathbb{R}^n$ satisfies the $p$-capacity thick complement  condition~\eqref{hyp:P} with two constants $c_0$, $r_0 \in \mathbb{R}^+$. Then one can find $\gamma = \gamma(n,p,\Lambda_1,\Lambda_2) >1$ such that for every $0 \le \alpha < \frac{n}{\gamma}$, $0<q<\frac{n \gamma}{n-\alpha \gamma}$ and $0<s \le \infty$, there holds
\begin{align}\label{eq:regularityMalpha}
\|{\mathbf{M}}_\alpha(|\nabla u|^p)\|_{L^{q,s}(\Omega)} \leq C \|\mathbf{M}_\alpha(|\mathcal{E}|^p)\|_{L^{q,s}(\Omega)},
\end{align}
where $C$ depends only on $n,p,\Lambda_1,\Lambda_2,\alpha,c_0,r_0,\mathrm{diam}(\Omega),q,s$. 
\end{theorem}
\begin{theorem}[Good-$\lambda$ theorem under assumption~\eqref{hyp:R}]
\label{theo:good-lambda-Rf} 
For any $\alpha \in [0,n)$ and $\varepsilon \in \mathbb{R}^+$, one can find some constants $\tilde{p} = \tilde{p}(n,p,\Lambda_1,\Lambda_2)>p$,  $\delta = \delta(n,p,\alpha,\varepsilon) \in (0,1)$, $a_{\infty} = a_{\infty}(n,p,\alpha)>1$ and $a_{\varepsilon} = a_{\varepsilon}(n,p,\alpha,\varepsilon) \in (0,1)$ such that if $\Omega$ is a $(\delta,r_0)$-Reifenberg flat domain~\eqref{hyp:R} satisfying $[\mathbb{A}]_{\tilde{p}}^{r_0} \le \delta$ for some $r_0 \in \mathbb{R}^+$ then
\begin{multline}\label{eq:good-lambda-Rf}
\mathcal{L}^n\left(\left\{\zeta \in \Omega: \ {\mathbf{M}}_{\alpha}(|\nabla u|^p)(\zeta) > a_{\infty} \lambda, \ {\mathbf{M}}_{\alpha}(|\mathcal{E}|^p)(\zeta) \le a_{\varepsilon} \lambda \right\} \right) \\
 \hspace{4cm} \le C \varepsilon \mathcal{L}^n\left(\left\{\zeta \in \Omega: \ {\mathbf{M}}_{\alpha} (|\nabla u|^p)(\zeta) > \lambda\right\}\right),
\end{multline}
holds for all $\lambda \in \mathbb{R}^+$, where $C = C(n, p, \Lambda_1, \Lambda_2,  \alpha, r_0, \mathrm{diam}(\Omega))>0$.
\end{theorem}
\begin{theorem}[Global Lorentz regularity estimate under assumption~\eqref{hyp:R}]
\label{theo:main-Rf}
For any $\alpha \in [0,n)$, one can find $\tilde{p} = \tilde{p}(n,p,\Lambda_1,\Lambda_2)>p$ and  $\delta = \delta(n,p,\alpha) \in (0,1)$ such that if $\Omega$ is a $(\delta,r_0)$-Reifenberg flat domain~\eqref{hyp:R} satisfying $[\mathbb{A}]_{\tilde{p}}^{r_0} \le \delta$ for some $r_0 \in \mathbb{R}^+$ then~\eqref{eq:regularityMalpha} holds for every $0<q<\infty$ and $0<s \le \infty$. 
\end{theorem}

Let us discuss some comments on the results stated above. According to Remark~\ref{rem:2types}, hypothesis~\eqref{hyp:R} together with the small BMO condition on $\mathbb{A}$ in~\eqref{cond:BMO} ensures the better regularity results of solutions than ones under hypothesis~\eqref{hyp:P}. The ultimate essence of this difference is the ability to cross the boundary of domain when $\Omega$ satisfies~\eqref{hyp:P} or~\eqref{hyp:R}, separately. Theorems~\ref{theo:main} and~\ref{theo:main-Rf} are in fact consequence of ``good-$\lambda$'' results in Theorems \ref{theo:good-lambda} and \ref{theo:good-lambda-Rf} over measuring sets, with different constraints on the boundary of $\Omega$ confirm higher regularity results regarding the elliptic equations involving the $p$-Laplacian. More particularly, a noteworthy feature of Theorem~\ref{theo:main-Rf} is that the scales $q$ and $s$ vary in the range specified in~\eqref{eq:regularityMalpha}, it implies that the critical exponent $q$ of Lorentz spaces $L^{q,s}(\Omega)$ can be enlarged. This result extends to a broader range of $q$, compared to the one in Theorem~\ref{theo:main}. \\

1.4. \textbf{Paper Overview.} The rest of the paper is organized as follows. In the next section, we state and prove some preliminary results needed in the paper, in which both local and boundary comparison estimates in Lebesgue spaces will be obtained. Section~\ref{sec:levelset} deals with some level-set inequalities that are important and essential in our sophisticated proofs later. Towards this goal, the hypotheses of the Covering Lemma~\ref{lem:mainlem} are satisfied. Finally, Section~\ref{sec:proofs} is dedicated to the proofs of main theorems.

\section{Lebesgue comparison estimates}\label{sec:inter_bound}
In this section, we first recall a global estimate in Lebesgue space. We then establish some local comparison estimates between the weak solution to problem~\eqref{eq:main} in a ball with solutions to corresponding homogeneous problems in the smaller balls. 
\subsection{Global estimate}
\begin{lemma}\label{lem:global}
There exists a positive constant $C = C(n,p,\Lambda_1,\Lambda_2)$ such that
\begin{equation}\label{eq:prop1}
\int_{\Omega} |\nabla u|^p dx \le C \int_{\Omega} |\mathcal{E}|^p dx.
\end{equation}
\end{lemma}
\begin{proof}
The variational formula of equation~\eqref{eq:main} is given by
\begin{equation*}
\int_{\Omega} \langle \mathbb{A}(x,\nabla u), \nabla \varphi \rangle dx = \int_{\Omega} \langle \mathbb{B}(x, \mathbf{F}), \nabla \varphi \rangle dx - \int_{\Omega} f \varphi dx,
\end{equation*}
for all test functions $\varphi$ in Sobolev space $W^{1,p}_0(\Omega)$. By taking $\varphi = u - g \in W^{1,p}_0(\Omega)$ as a test function of equation~\eqref{eq:main}, it follows that
\begin{equation*}
\int_{\Omega} \langle \mathbb{A}(x,\nabla u), \nabla u \rangle dx = \int_{\Omega} \langle \mathbb{A}(x,\nabla u), \nabla  g \rangle dx + \int_{\Omega} \langle \mathbb{B}(x,\mathbf{F}), \nabla u - \nabla g \rangle dx  - \int_{\Omega} f (u - g) dx.
\end{equation*}
Thanks to assumptions~\eqref{cond:A1} and~\eqref{cond:A2} on operators $\mathbb{A}$ and $\mathbb{B}$, one deduces from the above equation that
\begin{align}\label{ineq:S-12}
 \int_{\Omega} |\nabla u|^p dx & \le \Lambda_2^{-1}(\Lambda_1+1)  \left(\mathcal{S}_1 +  \mathcal{S}_2 \right),
\end{align}
where $\mathcal{S}_1$ and $\mathcal{S}_2$ are given by
\begin{align*}
\mathcal{S}_1 =  \int_{\Omega} |\nabla u|^{p-1} |\nabla g| dx + \int_{\Omega} |\mathbf{F}|^{p-1} \left(|\nabla u| + |\nabla g|\right) dx, 
\end{align*}
and
\begin{align*}
\mathcal{S}_2 = \int_{\Omega} |f||u-g| dx.
\end{align*}
The first term $\mathcal{S}_1$ can be estimated by applying H{\"o}lder and Young's inequalities. It is easily to check 
\begin{align}\nonumber
\mathcal{S}_1 & \le  \left( \int_{\Omega} |\nabla u|^{p} dx \right)^{\frac{p-1}{p}}\left( \int_{\Omega} |\nabla g|^p dx \right)^{\frac{1}{p}} + \left( \int_{\Omega} |\mathbf{F}|^{p} dx \right)^{\frac{p-1}{p}}\left( \int_{\Omega} |\nabla u|^p dx \right)^{\frac{1}{p}} \\ \nonumber
 & \hspace{3cm} +  \left( \int_{\Omega} |\mathbf{F}|^{p} dx \right)^{\frac{p-1}{p}}\left( \int_{\Omega} |\nabla g|^p dx \right)^{\frac{1}{p}}  \\ \label{ineq:S_1}
 & \le \varepsilon \int_{\Omega} |\nabla u|^{p} dx  + c(\varepsilon) \int_{\Omega} |\mathbf{F}|^{p} dx  + c(\varepsilon) \int_{\Omega} |\nabla g|^{p} dx,
\end{align}
for every positive number $\varepsilon$. Since $1 < p < n$, Sobolev's inequality gives us
\begin{align}\label{ineq:Sobolev}
\left(\int_{\Omega} |u-g|^{p^*}dx\right)^{\frac{1}{p^*}} \le C \left(\int_{\Omega} |\nabla u - \nabla g|^p dx \right)^{\frac{1}{p}},
\end{align}
for all $1 \le p^* \le \frac{np}{n-p}$. Thanks to H{\"o}lder's inequality and applying~\eqref{ineq:Sobolev} with $p^* = p$, we can estimate the next term $\mathcal{S}_2$ as follows
\begin{align}\nonumber
\mathcal{S}_2 & \le  \left(\int_{\Omega} |u-g|^{p}dx\right)^{\frac{1}{p}} \left(\int_{\Omega} |f|^{\frac{p}{p-1}} dx\right)^{1 - \frac{1}{p}} \\ \nonumber
&\le C \left(\int_{\Omega} |\nabla u - \nabla g|^{p} dx\right)^{\frac{1}{p}} \left(\int_{\Omega} |f|^{\frac{p}{p-1}} dx\right)^{1 - \frac{1}{p}}   \\ \label{ineq:S_2}
& \le \varepsilon \int_{\Omega} \left( |\nabla u|^p + |\nabla g|^p \right) dx +  c(\varepsilon) \int_{\Omega} |f|^{\frac{p}{p-1}} dx.
\end{align} 
Then,~\eqref{eq:prop1} can be handled by taking into account~\eqref{ineq:S_1} and~\eqref{ineq:S_2} to~\eqref{ineq:S-12}, with a suitable choice of $\varepsilon = \frac{1}{4}{(\Lambda_1+1)}^{-1}{\Lambda_2}$.
\end{proof}

\subsection{Local comparison estimates}
Let us fix a point $\xi \in \overline{\Omega}$ and $0<2\varrho \le r_0$, where $\Omega$ satisfies~\eqref{hyp:P} with two positive constants $c_0$ and $r_0$. We recall here notation $\Omega_{2\varrho}(\xi)$ denotes the set $B_{2\varrho}(\xi) \cap \Omega$. 
 
\begin{lemma}\label{lem:u-w}
Assume that $v \in W^{1,p}(\Omega_{2\varrho}(\xi))$ is the unique solution to the following equation:
\begin{equation}\label{eq:I1}
\begin{cases} \mbox{div} \left(\mathbb{A}(x,\nabla v)\right) & = \ 0, \quad \ \quad \mbox{ in } {\Omega_{2\varrho}(\xi)},\\ 
\hspace{1.2cm} v & = \ u - g, \ \mbox{ on } \partial {\Omega_{2\varrho}(\xi)}.\end{cases}
\end{equation}
Then there are two constants $\gamma = \gamma(n,p,\Lambda_1,\Lambda_2)>1$ and $C = C(n,p,\Lambda_1,\Lambda_2)>0$ such that 
\begin{equation}\label{ineq:RH}
\left(\fint_{\Omega_{\varrho}(\xi)}|\nabla v|^{\gamma p} dx\right)^{\frac{1}{\gamma p}}\leq C\left(\fint_{\Omega_{2\varrho}(\xi)}|\nabla v|^p dx\right)^{\frac{1}{p}}.
\end{equation}
Moreover for all $\delta \in (0,1)$ there holds
\begin{align}\label{ineq:lem:u-w}
\fint_{{\Omega_{2\varrho}(\xi)}} |\nabla u - \nabla v|^p dx  & \le \delta \fint_{{\Omega_{2\varrho}(\xi)}} |\nabla u|^pdx + C {\delta}^{1-\theta } \fint_{{\Omega_{2\varrho}(\xi)}} |\mathcal{E}|^p dx,   
\end{align}
where the constant $\theta $ is determined by
\begin{align}\label{def:beta}
\theta  = \begin{cases} p, & \ \mbox{ if }  \ 2 \le p \le n, \\ \frac{2}{p-1}, & \ \mbox{ if } \ 1< p <2.\end{cases}
\end{align}
\end{lemma}
\begin{proof}
The first inequality~\eqref{ineq:lem:u-w} is well-known as a type of Gehring's lemma or as a kind of ``reverse'' H\"older inequality which has been investigated in~\cite[Theorem 6.7]{Giu}. We also refer the reader to several papers such as~\cite{Mi3, Phuc1, 55QH4, MPT2018, HOk2019} and references therein for the proof of this inequality. Therefore, to prove the local comparison~\eqref{ineq:lem:u-w} is our major goal here. Let us first subtract the variational formulas of equations~\eqref{eq:main} and~\eqref{eq:I1}, that yields
\begin{equation*}
\int_{\Omega} \langle \mathbb{A}(x,\nabla u) - \mathbb{A}(x,\nabla v), \nabla \varphi \rangle dx = \int_{\Omega} \langle \mathbb{B}(x, \mathbf{F}), \nabla \varphi \rangle dx - \int_{\Omega} f \varphi dx,
\end{equation*}
for all $\varphi \in W^{1,p}_0(\Omega)$. By taking $\varphi = u - v - g \in W^{1,p}_0(\Omega_{2\varrho}(\xi))$ as a test function of equation~\eqref{eq:main}, it follows that
\begin{align*}
\fint_{{\Omega_{2\varrho}(\xi)}} & \langle \mathbb{A}(x,\nabla u) - \mathbb{A}(x,\nabla v), \nabla u -\nabla v \rangle dx \\
& = \fint_{{\Omega_{2\varrho}(\xi)}} \langle \mathbb{A}(x,\nabla u) - \mathbb{A}(x,\nabla v), \nabla  g \rangle dx   +  \fint_{{\Omega_{2\varrho}(\xi)}} \langle \mathbb{B}(x,\mathbf{F}), \nabla u - \nabla v \rangle dx \\ 
& \qquad - \fint_{{\Omega_{2\varrho}(\xi)}} \langle \mathbb{B}(x,\mathbf{F}), \nabla g \rangle dx  - \fint_{{\Omega_{2\varrho}(\xi)}} f (u - v - g) dx.
\end{align*}
Assumptions~\eqref{cond:A1} and~\eqref{cond:A2} of operator $\mathbb{A}$ give that
\begin{align}\label{est:u-w-I}
\Lambda_2 \fint_{{\Omega_{2\varrho}(\xi)}} {\Phi}(\nabla u,\nabla v) dx & \le (\Lambda_1+1) \left(\mathcal{S}_1 + \mathcal{S}_2  + \mathcal{S}_3 \right), 
\end{align}
where ${\Phi}$ is defined in~\eqref{def:G} and $\mathcal{S}_1$, $\mathcal{S}_2$, $\mathcal{S}_3$ are given respectively by
\begin{align*}
& \mathcal{S}_1 = \fint_{{\Omega_{2\varrho}(\xi)}} \left( |\nabla u|^{p-1} + |\nabla v|^{p-1}\right)|\nabla g| dx, \\ 
& \mathcal{S}_2 = \fint_{{\Omega_{2\varrho}(\xi)}} |\mathbf{F}|^{p-1} |\nabla u - \nabla v| dx +  \fint_{{\Omega_{2\varrho}(\xi)}} |\mathbf{F}|^{p-1} |\nabla g| dx, \\ 
& \mathcal{S}_3 = \fint_{{\Omega_{2\varrho}(\xi)}} |f||u - v -g|dx. 
\end{align*}
The first term $\mathcal{S}_1$ can be estimated by H{\"o}lder's and Young's inequalities with the following remark
$$|\nabla v|^{p-1} \le C(p) (|\nabla u|^{p-1} + |\nabla u - \nabla v|^{p-1}).$$ 
More precisely, for all $\varepsilon>0$, there holds
\begin{align}\nonumber 
\mathcal{S}_1 & \le C \fint_{{\Omega_{2\varrho}(\xi)}} (|\nabla u|^{p-1} + |\nabla u - \nabla v|^{p-1}) |\nabla g| dx \\ \nonumber 
& \le \frac{\varepsilon}{3} \fint_{{\Omega_{2\varrho}(\xi)}} |\nabla u - \nabla v|^p dx + C \varepsilon^{1-p} \fint_{{\Omega_{2\varrho}(\xi)}} |\nabla g|^p dx  \\ \label{est:S1}
& \hspace{3cm} + C\left(\fint_{{\Omega_{2\varrho}(\xi)}} |\nabla u|^p dx \right)^{\frac{p-1}{p}} \left(\fint_{{\Omega_{2\varrho}(\xi)}} |\nabla g|^p dx \right)^{\frac{1}{p}}.
\end{align}
We can obtain the estimate for $\mathcal{S}_2$ in a similar way as follows
\begin{align}\label{est:S2}
\mathcal{S}_2 \le \frac{\varepsilon}{3} \fint_{{\Omega_{2\varrho}(\xi)}} |\nabla u - \nabla v|^p dx + \frac{\varepsilon}{3} \fint_{{\Omega_{2\varrho}(\xi)}} |\nabla g|^p dx + C \varepsilon^{\frac{1}{1-p}} \fint_{{\Omega_{2\varrho}(\xi)}} |\mathbf{F}|^p  dx.
\end{align}
There remains to estimate the last term $\mathcal{S}_3$, and here we use the Sobolev's inequality
\begin{align}\nonumber 
\left(\fint_{{\Omega_{2\varrho}(\xi)}}|u-w-g|^{p^*}dx\right)^{\frac{1}{p^*}} & = \left(\frac{1}{\mathcal{L}^n({\Omega_{2\varrho}(\xi)})}\int_{{\Omega_{2\varrho}(\xi)}}|u-w-g|^{p^*}dx\right)^{\frac{1}{p^*}} \\ \nonumber
& \le C\left(\frac{1}{\mathcal{L}^n({\Omega_{2\varrho}(\xi)})}\right)^{\frac{1}{p^*}} \left(\int_{{\Omega_{2\varrho}(\xi)}}|\nabla u- \nabla v - \nabla g|^{p}dx\right)^{\frac{1}{p}} \\ \label{ineq:Sobolev-2}
& = C\left(\mathcal{L}^n({\Omega_{2\varrho}(\xi)})\right)^{\frac{1}{p} - \frac{1}{p^*}} \left(\fint_{{\Omega_{2\varrho}(\xi)}}|\nabla u - \nabla v - \nabla g|^{p}dx\right)^{\frac{1}{p}},
\end{align}
for all $1 \le p^* \le \frac{np}{n-p}$. Then, H{\"o}lder's inequality and~\eqref{ineq:Sobolev-2} are applied with $p^* = p$, one has
\begin{align}\nonumber
\mathcal{S}_3 &\le   \left(\fint_{{\Omega_{2\varrho}(\xi)}} |u-w-g|^{p}dx\right)^{\frac{1}{p}} \left(\fint_{{\Omega_{2\varrho}(\xi)}} |f|^{\frac{p}{p-1}} dx\right)^{1 - \frac{1}{p}} \\ \nonumber
& \le C \left(\fint_{{\Omega_{2\varrho}(\xi)}} |\nabla u - \nabla v - \nabla g|^{p} dx\right)^{\frac{1}{p}} \left(\fint_{{\Omega_{2\varrho}(\xi)}} |f|^{\frac{p}{p-1}} dx\right)^{1 - \frac{1}{p}}   \\ \nonumber
& \le \frac{\varepsilon}{3} \fint_{{\Omega_{2\varrho}(\xi)}} 2^{-p}  |\nabla u - \nabla v - \nabla g|^p  dx +  C \varepsilon^{\frac{1}{1-p}} \fint_{{\Omega_{2\varrho}(\xi)}} |f|^{\frac{p}{p-1}} dx \\ \label{est:S_3}
& \le \frac{\varepsilon}{3} \fint_{{\Omega_{2\varrho}(\xi)}} |\nabla u - \nabla v|^p  dx + \frac{\varepsilon}{3} \fint_{{\Omega_{2\varrho}(\xi)}} |\nabla g|^p  dx +  C \varepsilon^{\frac{1}{1-p}} \fint_{{\Omega_{2\varrho}(\xi)}} |f|^{\frac{p}{p-1}} dx.
\end{align}
Combining~\eqref{est:S1},~\eqref{est:S2} and~\eqref{est:S_3}, one deduces from~\eqref{est:u-w-I} that
\begin{align}\nonumber
\fint_{{\Omega_{2\varrho}(\xi)}} {\Phi}(\nabla u,\nabla v) dx & \le C \left[\varepsilon \fint_{{\Omega_{2\varrho}(\xi)}} |\nabla u - \nabla v|^p dx  + \left( \varepsilon^{1-p} + \varepsilon^{\frac{1}{1-p}} \right) \fint_{{\Omega_{2\varrho}(\xi)}} |\mathcal{E}|^p dx  \right. \\ \label{est:G1}
& \qquad \left. + \left(\fint_{{\Omega_{2\varrho}(\xi)}} |\nabla u|^p dx \right)^{\frac{p-1}{p}} \left(\fint_{{\Omega_{2\varrho}(\xi)}} |\nabla g|^p dx \right)^{\frac{1}{p}} \right],
\end{align}
for all $\varepsilon \in (0,1)$. On the other hand, it is easily to see that if $p \ge 2$ there holds
\begin{align*}
|z_1 - z_2|^p = |z_1 - z_2|^{p-2}|z_1 -z_2| \le {\Phi}(z_1,z_2), \quad \forall z_1, \, z_2 \in \mathbb{R}^n,
\end{align*}
which implies from~\eqref{est:G1} that
\begin{align}\nonumber
\fint_{{\Omega_{2\varrho}(\xi)}} |\nabla u - \nabla v|^p dx & \le C \left[\varepsilon \fint_{{\Omega_{2\varrho}(\xi)}} |\nabla u - \nabla v|^p dx  + \left(\varepsilon^{1-p}+\varepsilon^{\frac{1}{1-p}}\right) \fint_{{\Omega_{2\varrho}(\xi)}} |\mathcal{E}|^p dx \right. \\ \label{est:G2}
& \qquad \left. + \left(\fint_{{\Omega_{2\varrho}(\xi)}} |\nabla u|^p dx \right)^{\frac{p-1}{p}} \left(\fint_{{\Omega_{2\varrho}(\xi)}} |\nabla g|^p dx \right)^{\frac{1}{p}} \right].
\end{align}
We can choose a fixed number $\varepsilon \in (0,1)$ such that $C\varepsilon < \frac{1}{2}$ in~\eqref{est:G2} to get
\begin{align}\nonumber
\fint_{{\Omega_{2\varrho}(\xi)}} |\nabla u - \nabla v|^p dx & \le C \left[\fint_{{\Omega_{2\varrho}(\xi)}} |\mathcal{E}|^p dx \right. \\ \nonumber
& \qquad \qquad \left. + \left(\fint_{{\Omega_{2\varrho}(\xi)}} |\nabla u|^p dx \right)^{\frac{p-1}{p}} \left(\fint_{{\Omega_{2\varrho}(\xi)}} |\nabla g|^p dx \right)^{\frac{1}{p}} \right],
\end{align} 
and due to Young's inequality, for all $\delta \in (0,1)$, there holds
\begin{align}\nonumber
\fint_{{\Omega_{2\varrho}(\xi)}} |\nabla u - \nabla v|^p dx & \le \delta\fint_{{\Omega_{2\varrho}(\xi)}} |\nabla u|^p dx + C \delta^{1-p} \fint_{{\Omega_{2\varrho}(\xi)}} |\mathcal{E}|^p dx.
\end{align}
This guarantees~\eqref{ineq:lem:u-w} with $\theta  = p$. Otherwise, if $1 <p <2$ then we may estimate $|z_1 - z_2|^p$ for all $z_1$, $z_2 \in \mathbb{R}^n$ as follows
\begin{align*}
|z_1 - z_2|^p = (|z_1|^2 + |z_2|^2)^{\frac{p(2-p)}{4}}[{\Phi}(z_1,z_2)]^{\frac{p}{2}} \le 4^p \left(|z_1|^p + |z_1 - z_2|^p\right)^{\frac{2-p}{2}} [{\Phi}(z_1,z_2)]^{\frac{p}{2}},
\end{align*}
which implies from H{\"o}lder's inequality that
\begin{align}\nonumber
\fint_{{\Omega_{2\varrho}(\xi)}} |\nabla u - \nabla v|^p dx & \le 4^p \fint_{{\Omega_{2\varrho}(\xi)}} (|\nabla u|^p + |\nabla u - \nabla v|^p)^{\frac{2-p}{2}} [{\Phi}(\nabla u,\nabla v)]^{\frac{p}{2}}  dx \\ \nonumber
& \le 4^p \left(\fint_{{\Omega_{2\varrho}(\xi)}} (|\nabla u|^p + |\nabla u - \nabla v|^p)dx\right)^{\frac{2-p}{2}} \left(\fint_{{\Omega_{2\varrho}(\xi)}} {\Phi}(\nabla u,\nabla v) dx\right)^{\frac{p}{2}}.
\end{align}
For every $\delta \in (0,1)$, applying Young's inequality on the right hand side of above inequality, we obtain that
\begin{align}\nonumber
\fint_{{\Omega_{2\varrho}(\xi)}} |\nabla u - \nabla v|^p dx & \le \frac{\delta}{4} \fint_{{\Omega_{2\varrho}(\xi)}} (|\nabla u|^p + |\nabla u - \nabla v|^p)dx \\ \nonumber
& \hspace{3cm} + C \delta^{\frac{p-2}{p}} \fint_{{\Omega_{2\varrho}(\xi)}} {\Phi}(\nabla u,\nabla v) dx,
\end{align}
which deduces from~\eqref{est:G1} that
\begin{align}\nonumber
\fint_{{\Omega_{2\varrho}(\xi)}} |\nabla u - \nabla v|^p dx & \le \frac{\delta}{4} \fint_{{\Omega_{2\varrho}(\xi)}} |\nabla u|^p dx + \left(\frac{\delta}{4} + C \delta^{\frac{p-2}{p}} \varepsilon\right) \fint_{{\Omega_{2\varrho}(\xi)}} |\nabla u - \nabla v|^p dx \\ \nonumber
& \qquad +   C  \delta^{\frac{p-2}{p}} \left(\varepsilon^{1-p}+ \varepsilon^{\frac{1}{1-p}}\right) \fint_{{\Omega_{2\varrho}(\xi)}} |\mathcal{E}|^p dx \\ \label{est:G3a}
& \qquad +  C\delta^{\frac{p-2}{p}} \left(\varepsilon \fint_{{\Omega_{2\varrho}(\xi)}} |\nabla u|^p dx \right)^{\frac{p-1}{p}} \left(\varepsilon^{1-p}\fint_{{\Omega_{2\varrho}(\xi)}} |\nabla g|^p dx \right)^{\frac{1}{p}}.
\end{align}
Remark that for $\varepsilon \in (0,1)$ and $p \in (1,2)$ then one has $\varepsilon^{1-p} \le \varepsilon^{\frac{1}{1-p}}$. Applying H{\"o}lder's inequality for the last term on the right hand side of~\eqref{est:G3a}, one gets that
\begin{align}\nonumber
\fint_{{\Omega_{2\varrho}(\xi)}} |\nabla u - \nabla v|^p dx & \le \left(\frac{\delta}{4} + C \delta^{\frac{p-2}{p}} \varepsilon\right) \fint_{{\Omega_{2\varrho}(\xi)}} |\nabla u|^p dx \\ \nonumber
& \qquad + \left(\frac{\delta}{4} + C \delta^{\frac{p-2}{p}} \varepsilon\right) \fint_{{\Omega_{2\varrho}(\xi)}} |\nabla u - \nabla v|^p dx \\ \label{est:G3}
& \qquad +   C  \delta^{\frac{p-2}{p}} \varepsilon^{\frac{1}{1-p}} \fint_{{\Omega_{2\varrho}(\xi)}} |\mathcal{E}|^p dx.
\end{align}
For every $\delta \in (0,1)$, by taking $\varepsilon = (4C)^{-1} \delta^{\frac{2}{p}}$ in~\eqref{est:G3}, one obtains~\eqref{ineq:lem:u-w} which completes the proof.
\end{proof}

In the following lemma, by adding an extra condition, in which the nonlinear operator $\mathbb{A}$ satisfies the small BMO condition as in \eqref{cond:BMO}, we shall prove the comparison estimate between the weak solution $u$ to~\eqref{eq:main} and solution $w$ (take a note that we consider one further reference problem in \eqref{eq:w-I2} admits solution $w$).

\begin{lemma}\label{lem:B1}
Let $v$ be the unique solution to~\eqref{eq:I1} 
and $w$ be the unique solution to the following equation 
\begin{equation}\label{eq:w-I2}
\begin{cases} \mathrm{div} \left( \overline{\mathbb{A}}_{\Omega_{\varrho}(\xi)}(\nabla w)\right) & = \ 0, \ \mbox{ in } \Omega_{\varrho}(\xi),\\ 
\hspace{1.2cm} w & = \ v, \ \mbox{ on } \partial \Omega_{\varrho}(\xi).\end{cases}
\end{equation}
Assume that $[\mathbb{A}]_{\tilde{p}}^{\varrho} < \infty$ for $\tilde{p} = \frac{p\gamma}{\gamma - 1}$, where $\gamma$ is given in Lemma~\ref{lem:u-w}. Then there exists a constant $C = C(n,p,\Lambda_1,\Lambda_2)$ such that 
\begin{align}\label{eq:theoIc}
\|\nabla w\|^p_{L^{\infty}(\Omega_{\varrho/2}(\xi))} & \le C \left( \fint_{\Omega_{2\varrho}(\xi)} |\nabla u|^p  dx +    \fint_{\Omega_{2\varrho}(\xi)}  |\mathcal{E}|^p dx \right),  
\end{align}
and
\begin{multline}\label{eq:theoId}
\fint_{\Omega_{\varrho}(\xi)} |\nabla u - \nabla w|^pdx \le C\left( \left([\mathbb{A}]_{\tilde{p}}^{\varrho}\right)^p \fint_{\Omega_{2\varrho}(\xi)} |\nabla u|^p dx  + \left([\mathbb{A}]_{\tilde{p}}^{\varrho}\right)^{p(1-\theta )}  \fint_{\Omega_{2\varrho}(\xi)} |\mathcal{E}|^p dx \right).
\end{multline} 
\end{lemma}
\begin{proof}
We firstly refer the reader to~\cite{MP11} for the proof of two following inequalities
\begin{equation}\label{eq:lem2b}
C^{-1}\fint_{\Omega_{2\varrho}(\xi)} |\nabla w|^p dx \le \fint_{\Omega_{2\varrho}(\xi)} |\nabla v|^p dx \le C \fint_{\Omega_{2\varrho}(\xi)} |\nabla w|^p dx,
\end{equation}
and
\begin{equation}\label{eq:lem2a}
\fint_{\Omega_{\varrho}(\xi)} |\nabla v - \nabla w|^p dx \le C \left([\mathbb{A}]^{\varrho}_{\tilde{p}}\right)^p \fint_{\Omega_{2\varrho}(\xi)} |\nabla v|^p dx.
\end{equation}
Thanks to~\eqref{eq:lem2b} and~\eqref{ineq:lem:u-w} in Lemma~\ref{lem:u-w}, for all $\delta \in (0,1)$ there holds
\begin{align} \nonumber
\|\nabla w\|^p_{L^{\infty}(\Omega_{\varrho/2}(\xi))} &\le C \fint_{\Omega_{2\varrho}(\xi)} |\nabla w|^p dx   \le C \fint_{\Omega_{2\varrho}(\xi)} |\nabla v|^p dx \\ \nonumber  
& \le C \left( \fint_{\Omega_{2\varrho}(\xi)} |\nabla u|^p  dx +  \fint_{\Omega_{2\varrho}(\xi)} |\nabla u - \nabla v|^p dx \right) \\ \label{eq:theoId_1}
& \le C \left( \fint_{\Omega_{2\varrho}(\xi)} |\nabla u|^p  dx +   \delta^{1-\theta } \fint_{\Omega_{2\varrho}(\xi)}  |\mathcal{E}|^p dx \right),
\end{align}
which guarantees~\eqref{eq:theoIc} by choosing a fixed $\delta$. On the other hand, from~\eqref{eq:lem2a} one gets that
\begin{align*}
\fint_{\Omega_{\varrho}(\xi)} |\nabla u - \nabla w|^p dx &\le C\left( \fint_{\Omega_{\varrho}(\xi)} |\nabla u - \nabla v|^p dx +  \fint_{\Omega_{\varrho}(\xi)} |\nabla v - \nabla w|^p dx \right) \\
& \le C\left( \fint_{\Omega_{2\varrho}(\xi)} |\nabla u - \nabla v|^p dx + \left([\mathbb{A}]_{\tilde{p}}^{\varrho}\right)^p \fint_{\Omega_{2\varrho}(\xi)} |\nabla v|^p dx \right),
\end{align*}
which deduces from~\eqref{eq:theoId_1} and~\eqref{ineq:lem:u-w} in Lemma~\ref{lem:u-w} that
\begin{align}\nonumber
\fint_{\Omega_{\varrho}(\xi)} |\nabla u - \nabla w|^p dx & \le C\left( \fint_{\Omega_{2\varrho}(\xi)} |\nabla u - \nabla v|^p dx + \left([\mathbb{A}]_{\tilde{p}}^{\varrho}\right)^p \fint_{\Omega_{2\varrho}(\xi)} |\nabla u|^p dx \right. \\ \nonumber
& \left. \qquad \qquad + \left([\mathbb{A}]_{\tilde{p}}^{\varrho}\right)^p \delta^{1-\theta } \fint_{\Omega_{2\varrho}(\xi)} |\mathcal{E}|^p dx \right) \\ \nonumber
& \le C\left( \left[\varepsilon + \left([\mathbb{A}]_{\tilde{p}}^{\varrho}\right)^p\right] \fint_{\Omega_{2\varrho}(\xi)} |\nabla u|^p dx \right. \\ \label{est:uw} & \left. \qquad \qquad + \left[ \varepsilon^{1-\theta } + \left([\mathbb{A}]_{\tilde{p}}^{\varrho}\right)^p \delta^{1-\theta }\right] \fint_{\Omega_{2\varrho}(\xi)} |\mathcal{E}|^p dx \right),
\end{align}
for all $\varepsilon$, $\delta \in (0,1)$. Let us take $\varepsilon = \left([\mathbb{A}]_{\tilde{p}}^{\varrho}\right)^p$ and $\delta = \varepsilon^{\frac{\theta }{\theta  - 1}}$ in~\eqref{est:uw}, we obtain~\eqref{eq:theoId}.
\end{proof}

\section{Level-set inequalities}\label{sec:levelset}
We emphasize that the key technique to prove the good-$\lambda$ inequalities for measuring sets in both Theorem~\ref{theo:good-lambda} and~\ref{theo:good-lambda-Rf} is based on applying the following well-known lemma. This lemma has its own role in measure theory, that colloquially discussed as the substitution of Calder\'on-Zygmund-Krylov-Safonov decomposition. This lemma is a version of Calder\'on-Zygmund (or Vitali type) covering lemma that allows us to work with balls instead of cubes, see~\cite[Lemma 4.2]{CC1995} or~\cite{Vitali08}.

The problem actually has two distinct hypotheses to consider, under \eqref{hyp:P} and \eqref{hyp:R}. Therefore, we state here the ``sandwich'' version combined both hypotheses on $\Omega$, for the sake of simplicity. Here we refer the interested reader to \cite{Phuc2015, SSB2, BW2, SSB4,Phuc1, Phuc2, 55QH4,MPT2018,PNCCM,MPTNsub,PNJDE,PNnonuniform}, where these assumptions are independently performed.

\begin{lemma}[Covering Lemma]\label{lem:mainlem}
Let $\Omega \subset \mathbb{R}^n$ be a bounded domain satisfying one of the two following assumptions:
\begin{itemize}
\item[(i)] Hypothesis \eqref{hyp:P} with two constants $c_0$, $r_0>0$;
\item[(ii)] Hypothesis \eqref{hyp:R}: $\Omega$ is a $(\delta,r_0)$-Reifenberg flat domain with $\delta>0$ small enough, for some $r_0>0$.
\end{itemize}

Let $\varepsilon \in (0,1)$ and two measurable sets $\mathcal{V}\subset \mathcal{W} \subset\Omega$ satisfying $\mathcal{L}^n\left(\mathcal{V}\right) \le \varepsilon \mathcal{L}^n\left(B_{R}(0)\right)$ for some $R \in \left(0,\displaystyle{\frac{r_0}{20}}\right]$ (even if not concentric to $B_R$, by translating we can still assume with the ball centered at the origin). Suppose that for all $\xi \in \Omega$ and $\varrho \in (0,R]$, if $\mathcal{L}^n\left(\mathcal{V} \cap B_{\varrho}(\xi)\right) > \varepsilon \mathcal{L}^n\left(B_{\varrho}(\xi)\right)$ then $\Omega_{\varrho}(\xi)  \subset \mathcal{W}$. Then there exists a constant $C=C(n)>0$ such that $\mathcal{L}^n\left(\mathcal{V}\right)\leq C \varepsilon \mathcal{L}^n\left(\mathcal{W}\right)$.
\end{lemma} 

Notice that in the statement of Lemma \ref{lem:mainlem}, two hypotheses (i) and (ii) above are considered separately and independently, so we still use the letter $r_0$ in two appropriate conditions here with one symbol, to maintain simplicity and avoid confusion.

In order to check the the hypotheses of this covering lemma, we separate to several small steps and the following series of lemmas are essential.
\begin{lemma}\label{lem:A1} 
Let $\alpha \in [0,n)$ and $z_1 \in \Omega$ satisfying ${\mathbf{M}}_{\alpha}(|\mathcal{E}|^p)(z_1) \le \lambda_1$. Then for every $R>0$, there exists a constant $C = C(n,\alpha,\mathrm{diam}(\Omega)/R)>0$ such that 
\begin{align}\label{ineq:A1}
\mathcal{L}^n\left(\left\{\zeta \in \Omega: \  {\mathbf{M}}_{\alpha}(|\nabla u|^p)(\zeta)>\lambda_2 \right\} \right) \le C\left(\frac{\lambda_1}{\lambda_2}\right)^{\frac{n}{n-\alpha}} \mathcal{L}^n\left(B_R(0)\right), 
\end{align}
for all $\lambda_2>0$.
\end{lemma}
\begin{proof}
Thanks to Lemma~\ref{lem:M_alpha}, one gets that
\begin{align}\nonumber 
\mathcal{L}^n\left(\left\{\zeta \in \Omega: \  {\mathbf{M}}_{\alpha}(|\nabla u|^p)(\zeta)>\lambda_2 \right\} \right)  \le C \left(\frac{1}{\lambda_2}\int_{\Omega}{|\nabla u|^p dx}\right)^{\frac{n}{n-\alpha}},
\end{align}
which follows from the global estimate~\eqref{eq:prop1} in Lemma~\ref{lem:global} that
\begin{align}\label{est-3.2}
\mathcal{L}^n\left(\left\{\zeta \in \Omega: \  {\mathbf{M}}_{\alpha}(|\nabla u|^p)(\zeta)>\lambda_2 \right\} \right) \le C\left(\frac{1}{\lambda_2}\int_{\Omega}{|\mathcal{E}|^p dx}\right)^{\frac{n}{n-\alpha}}.
\end{align}
Let us denote the ball $Q : =  B_{D_0}(z_1)$ with $D_0 = \mathrm{diam}(\Omega)$. Recall that ${\mathbf{M}}_{\alpha}(|\mathcal{E}|^p)(z_1) \le \lambda_1$ and $\Omega \subset Q$, it implies from~\eqref{est-3.2} that
\begin{align*}
\mathcal{L}^n\left(\left\{\zeta \in \Omega: \  {\mathbf{M}}_{\alpha}(|\nabla u|^p)(\zeta)>\lambda_2 \right\} \right)  & \le  C\left(\frac{\mathcal{L}^n(Q)}{\lambda_2}\fint_{Q}{|\mathcal{E}|^p dx}\right)^{\frac{n}{n-\alpha}} \\  
& \le C \left(\frac{\mathcal{L}^n(Q)}{\lambda_2} D_0^{-\alpha}  {\mathbf{M}}_{\alpha}(|\mathcal{E}|^p)(z_1) \right)^{\frac{n}{n-\alpha}} \\
&  \le C \left(\frac{\lambda_1}{\lambda_2}\right)^{\frac{n}{n-\alpha}} D_0^n.
\end{align*}
This leads to conclude that the following estimate holds
\begin{align}\nonumber
\mathcal{L}^n\left(\left\{\zeta \in \Omega: \  {\mathbf{M}}_{\alpha}(|\nabla u|^p)(\zeta)>\lambda_2 \right\} \right)  & \le C  \left(\frac{D_0}{R}\right)^n \left(\frac{\lambda_1}{\lambda_2}\right)^{\frac{n}{n-\alpha}} \mathcal{L}^n(B_R(0)),
\end{align}
which completes the proof of~\eqref{ineq:A1}.
\end{proof}
\begin{lemma}\label{lem:A2}
Let $\alpha \in [0,n)$ and $z_2 \in \Omega_{\varrho}(\xi)$ satisfying ${\mathbf{M}}_{\alpha}(|\nabla u|^p)(z_2) \le \lambda$. Then for all $\delta \ge 3^{n+1}$  there holds 
\begin{align}\nonumber
& \mathcal{L}^n \left(\left\{\zeta \in \Omega_{\varrho}(\xi): \ {\mathbf{M}}_{\alpha}(|\nabla u|^p)(\zeta) > \delta \lambda\right\} \right) \\ \label{eq:res11}
& \hspace{3cm} \le \mathcal{L}^n \left(\left\{\zeta \in \Omega_{\varrho}(\xi): \ {\mathbf{M}}_{\alpha}(\chi_{B_{2\varrho}(\xi)} |\nabla u|^p)(\zeta) > \delta \lambda \right\} \right).
\end{align}
\end{lemma}
\begin{proof}
For any $\zeta \in B_{\varrho}(\xi)$, it is easily to see that
\begin{align}\label{eq:res9}
{\mathbf{M}}_{\alpha}(|\nabla u|^p)(\zeta) \le \max \left\{ {\mathbf{M}}^{\varrho}_{\alpha}(|\nabla u|^p)(\zeta) ; \  \mathbf{T}^{\varrho}_{\alpha} (|\nabla u|^p)(\zeta) \right\},
\end{align}
where 
$${\mathbf{M}}^{\varrho}_{\alpha}(|\nabla u|^p)(\zeta) = \sup_{0 < \varrho' < \varrho} \ (\varrho')^{\alpha}{\fint_{B_{\varrho'}(\zeta)}{|\nabla u|^p dx}},$$
and 
$$ \mathbf{T}^{\varrho}_{\alpha} (|\nabla u|^p)(\zeta)  = \sup_{\varrho' \ge \varrho} \ (\varrho')^{\alpha}{\fint_{B_{\varrho'}(\zeta)}{|\nabla u|^p dx}}.$$
From~\eqref{eq:res9}, the measure on the left hand side of~\eqref{eq:res11} can be decomposed as follows
\begin{align}\nonumber
& \mathcal{L}^n \left(\left\{\zeta \in \Omega_{\varrho}(\xi): \ {\mathbf{M}}_{\alpha}(|\nabla u|^p)(\zeta) > \delta \lambda\right\} \right) \\ \nonumber
& \hspace{3cm} \le \mathcal{L}^n \left(\left\{\zeta \in \Omega_{\varrho}(\xi): \ {\mathbf{M}}^{\varrho}_{\alpha}(|\nabla u|^p)(\zeta) > \delta \lambda \right\} \right) \\ \label{eq:res9a}
& \hspace{5cm} +  \mathcal{L}^n \left(\left\{\zeta \in \Omega_{\varrho}(\xi): \ {\mathbf{T}}^{\varrho}_{\alpha}(|\nabla u|^p)(\zeta) > \delta \lambda \right\} \right).
\end{align}
Moreover, we can see that $B_{\varrho'}(\zeta) \subset B_{2\varrho}(\xi)$ for all $\varrho' \in (0,\varrho)$. Therefore, one has
\begin{align}\label{eq:res9b}
{\mathbf{M}}^{\varrho}_{\alpha}(|\nabla u|^p)(\zeta) = \sup_{0 < \varrho' < \varrho} \ (\varrho')^{\alpha}{\fint_{B_{\varrho'}(\zeta)}{\chi_{B_{2\varrho}(\xi)} |\nabla u|^p dx}} \le {\mathbf{M}}_{\alpha}(\chi_{B_{2\varrho}(\xi)} |\nabla u|^p)(\zeta).
\end{align}
On the other hand, since $B_{\varrho'}(\zeta) \subset B_{3\varrho'}(z_2)$ for all $\varrho' \ge \varrho$, it follows that
\begin{align}\nonumber
\mathbf{T}^{\varrho}_{\alpha} (|\nabla u|^p)(\zeta) & = \sup_{\varrho' \ge \varrho} \ (\varrho')^{\alpha}{\fint_{B_{\varrho'}(\zeta)}{|\nabla u|^p dx}} \\ \nonumber
& \le  \sup_{\varrho' \ge \varrho} \ (\varrho')^{\alpha} \frac{\mathcal{L}^n(B_{3\varrho'}(z_2))}{\mathcal{L}^n(B_{\varrho'}(\zeta))} {\fint_{B_{3\varrho'}(z_2)}{|\nabla u|^p dx}} \\ \nonumber
& \le 3^{n-\alpha} \sup_{\varrho' \ge \varrho} \ (3\varrho')^{\alpha} {\fint_{B_{3\varrho'}(z_2)}{|\nabla u|^p dx}} \\ \label{eq:res10}
& \le 3^n {\mathbf{M}}_{\alpha}(|\nabla u|^p)(z_2).
\end{align}
Thanks to assumption ${\mathbf{M}}_{\alpha}(|\nabla u|^p)(z_2) \le \lambda$ and estimate~\eqref{eq:res10}, we may conclude that
\begin{align}\label{eq:res10b}
\mathcal{L}^n \left(\left\{\zeta \in \Omega_{\varrho}(\xi): \ {\mathbf{T}}^{\varrho}_{\alpha}(|\nabla u|^p)(\zeta) > \delta \lambda \right\} \right) = 0, \mbox{ for all } \delta>3^n.
\end{align}
Finally, the proof of~\eqref{eq:res11} can be completed by taking into account~\eqref{eq:res9a},~\eqref{eq:res9b} and~\eqref{eq:res10b}.
\end{proof}
\begin{lemma}\label{lem:A3} 
Let $\alpha \in [0,n)$. One can find some constants $\vartheta$, $\kappa$ and $\varepsilon_0>0$ depending on $n$, $p$, $\Lambda_1$, $\Lambda_2$ and $\alpha$ such that if $z_1 \in \Omega_{\varrho}(\xi)$ and $z_2 \in B_{\varrho}(\xi)$ satisfying
\begin{align}\label{eq:x2}
{\mathbf{M}}_{\alpha}(|\nabla u|^p)(z_1) \le \lambda \ \mbox{ and } \
{\mathbf{M}}_{\alpha}(|\mathcal{E}|^p)(z_2) \le \varepsilon^{\kappa}\lambda,
\end{align}
then the following inequality
\begin{align}\label{eq:iigoal}
\mathcal{L}^n \left(\left\{\zeta \in \Omega_{\varrho}(\xi): \ {\mathbf{M}}_{\alpha}(\chi_{B_{2\varrho}(\xi)} |\nabla u|^p)(\zeta)> \varepsilon^{-\vartheta}\lambda \right\} \right) \le \varepsilon\mathcal{L}^n\left(B_{\varrho}(\xi)\right),
\end{align}
holds for all $\lambda>0$ and $\varepsilon \in (0,\varepsilon_0)$.
\end{lemma}
\begin{proof}
In order to obtain~\eqref{eq:iigoal}, we now consider two cases when $\xi$ belongs to the interior domain $B_{4\varrho}(\xi) \subset \Omega$ and $\xi$ is close to the boundary $B_{4\varrho}(\xi) \cap \partial\Omega \neq \emptyset$. In the first case $B_{4\varrho}(\xi) \subset \Omega$, let $v$ be the unique solution to the following equation
\begin{equation}\nonumber 
\begin{cases} \mbox{div} \left( \mathbb{A}(x,\nabla v)\right) & = \ 0, \quad \ \quad \mbox{ in } B_{4\varrho}(\xi),\\ 
\hspace{1.2cm} v & = \ u, \qquad \ \mbox{ on } \partial B_{4\varrho}(\xi).\end{cases}
\end{equation}
Thanks to reverse H{\"o}lder's inequality~\eqref{ineq:RH} and the comparison estimate~\eqref{ineq:lem:u-w} in Lemma~\ref{lem:u-w}, there exists a constant $\gamma>1$ such that
\begin{align}\label{eq:revH-1}
\left(\fint_{B_{2\varrho}(\xi)}{|\nabla v|^{\gamma p}} dx\right)^{\frac{1}{\gamma}} \le C \fint_{B_{4\varrho}(\xi)}{|\nabla v|^p dx},
\end{align}
and
\begin{align} \label{eq:com-1}
\fint_{B_{4\varrho}(\xi)}{|\nabla u - \nabla v|^pdx} & \le \delta \fint_{B_{4\varrho}(\xi)}{|\nabla u|^p dx}  + C {\delta}^{1-\theta } \fint_{B_{4\varrho}(\xi)}{|\mathcal{E}|^pdx},
\end{align}
for all $\delta \in (0,1)$ and $\theta $ is given as in~\eqref{def:beta}. We can decompose as follows
\begin{align}\nonumber
L : & =  \mathcal{L}^n \left(\left\{\zeta \in \Omega_{\varrho}(\xi): \ {\mathbf{M}}_{\alpha}(\chi_{B_{2\varrho}(\xi)} |\nabla u|^p)(\zeta)> \varepsilon^{-\vartheta}\lambda \right\} \right) \\ \nonumber
& \le C  \mathcal{L}^n\left(\left\{\zeta \in  B_{\varrho}(\xi): \ {\mathbf{M}}_{\alpha}(\chi_{B_{2\varrho}(\xi)} |\nabla u - \nabla v|^p)(\zeta)> \varepsilon^{-\vartheta}\lambda \right\} \right) \\ \label{eq:estV-1}
& \hspace{2cm}  + C\mathcal{L}^n\left(\left\{\zeta \in  B_{\varrho}(\xi): \ {\mathbf{M}}_{\alpha}(\chi_{B_{2\varrho}(\xi)} |\nabla v|^p)(\zeta)> \varepsilon^{-\vartheta}\lambda \right\}\right). 
\end{align}
Applying Lemma~\ref{lem:M_alpha} for two terms on the right hand side of~\eqref{eq:estV-1} with $s = 1$ and $s = \gamma>1$  respectively, we obtain that
\begin{align}\nonumber
L &\le C \left(\frac{1}{\varepsilon^{-\vartheta}\lambda} \int_{B_{2\varrho}(\xi)}|\nabla u - \nabla v|^p dx\right)^{\frac{n}{n-\alpha}}  + C \left(\frac{1}{\left(\varepsilon^{-\vartheta}\lambda\right)^{\gamma}} \int_{B_{2\varrho}(\xi)} |\nabla v|^{\gamma p} dx \right)^{\frac{n}{n-\alpha \gamma}}.
\end{align}
Performing as a form of the integral average, this inequality leads to
\begin{align}\label{eq:estV-2a}
L &\le  \left(\frac{(4\varrho)^n}{\varepsilon^{-\vartheta}\lambda} \fint_{B_{4\varrho}(\xi)}|\nabla u - \nabla v|^p dx\right)^{\frac{n}{n-\alpha}}  + C \left(\frac{(2\varrho)^n}{\left(\varepsilon^{-\vartheta}\lambda\right)^{\gamma}} \fint_{B_{2\varrho}(\xi)} |\nabla v|^{\gamma p} dx \right)^{\frac{n}{n-\alpha \gamma}} .
\end{align}
Thanks to~\eqref{eq:revH-1} and~\eqref{eq:com-1}, one obtains from~\eqref{eq:estV-2a} that
\begin{align}\nonumber
L &\le  C \left(\frac{(4\varrho)^n}{\varepsilon^{-\vartheta}\lambda} \right)^{\frac{n}{n-\alpha}} \left(\delta \fint_{B_{4\varrho}(\xi)}{|\nabla u|^p dx}  + {\delta}^{1-\theta } \fint_{B_{4\varrho}(\xi)}{|\mathcal{E}|^pdx}\right)^{\frac{n}{n-\alpha}} \\ \label{eq:estV-2}
& \qquad \qquad + C \left(\frac{(2\varrho)^{\frac{n}{\gamma}}}{\varepsilon^{-\vartheta}\lambda} \fint_{B_{4\varrho}(\xi)} |\nabla v|^{p} dx \right)^{\frac{n \gamma}{n-\alpha \gamma}} .
\end{align}
It is not difficult to see that $B_{4\varrho}(\xi) \subset B_{5\varrho}(z_1) \cap B_{5\varrho}(z_2)$  which yields from~\eqref{eq:x2} that
\begin{align}\label{eq:com-2a}
\fint_{B_{4\varrho}(\xi)}{|\nabla u|^pdx} & \le  C \fint_{B_{5\varrho}(z_1)}{|\nabla u|^pdx} \le C (5\varrho)^{-\alpha} {\mathbf{M}}_{\alpha}(|\nabla u|^p)(z_1) \le C {\varrho}^{-\alpha} \lambda,
\end{align}
and
\begin{align}\label{eq:com-2b}
\fint_{B_{4\varrho}(\xi)}{|\mathcal{E}|^pdx} \le  C \fint_{B_{5\varrho}(z_2)}{|\mathcal{E}|^pdx} \le C (5\varrho)^{-\alpha} {\mathbf{M}}_{\alpha}(|\mathcal{E}|^p)(z_2) \le C {\varrho}^{-\alpha} \varepsilon^{\kappa} \lambda.
\end{align}
On the other hand, thanks to~\eqref{eq:com-1} again and notice that $1+\delta \le 2$, one has
\begin{align} \nonumber
\fint_{B_{4\varrho}(\xi)} |\nabla v|^p dx & \le C \left(\fint_{B_{4\varrho}(\xi)} |\nabla u|^p dx + \fint_{B_{4\varrho}(\xi)} |\nabla u - \nabla v|^p dx\right) \\ \label{eq:com-2c}
& \le C \left(\fint_{B_{4\varrho}(\xi)} |\nabla u|^p dx + {\delta}^{1-\theta }\fint_{B_{4\varrho}(\xi)} |\mathcal{E}|^p dx\right).
\end{align}
Collecting~\eqref{eq:com-2a},~\eqref{eq:com-2b} and~\eqref{eq:com-2c} into~\eqref{eq:estV-2}, it gives
\begin{align*}
L &\le  C \left(\frac{(4\varrho)^n}{\varepsilon^{-\vartheta}\lambda} \right)^{\frac{n}{n-\alpha}} \left(\delta {\varrho}^{-\alpha} \lambda  + {\delta}^{1-\theta } {\varrho}^{-\alpha} \varepsilon^{\kappa} \lambda \right)^{\frac{n}{n-\alpha}} \\
& \hspace{4cm}  + C \left(\frac{(2\varrho)^{\frac{n}{\gamma}}}{\varepsilon^{-\vartheta}\lambda} \left(  {\varrho}^{-\alpha} \lambda  + {\delta}^{1-\theta } {\varrho}^{-\alpha} \varepsilon^{\kappa} \lambda \right) \right)^{\frac{n \gamma}{n-\alpha \gamma}} .
\end{align*}
By direct computation this inequality, we obtain that
\begin{align}\label{est:L}
L &\le  C \left[ {\varepsilon}^{\frac{\vartheta n}{n-\alpha}} (\delta + {\delta}^{1-\theta }\varepsilon^{\kappa})^{\frac{n}{n-\alpha}} + {\varepsilon}^{\frac{\vartheta n \gamma}{n-\alpha \gamma}} (1 + {\delta}^{1-\theta }\varepsilon^{\kappa})^{\frac{n \gamma}{n-\alpha \gamma}} \right]  \varrho^n .
\end{align}
Let us take $\delta = \varepsilon^{\frac{\kappa}{\theta }}$ in~\eqref{est:L}, we may conclude that
\begin{align}\label{est:Lb}
L \le C \left(\varepsilon^{\left(\vartheta + \frac{\kappa}{\theta }\right)\frac{n}{n-\alpha}} + \varepsilon^{\frac{\vartheta n \gamma}{n-\alpha \gamma}} \right) \varrho^n \le  C \left(\varepsilon^{\left(\vartheta + \frac{\kappa}{\theta }\right)\frac{n}{n-\alpha}} + \varepsilon^{\frac{\vartheta n \gamma}{n-\alpha \gamma}} \right) \mathcal{L}^n \left(B_{\varrho}(\xi)\right).
\end{align}
To balance the exponent of $\varepsilon$ on the right hand side of~\eqref{est:Lb}, let us choose $\vartheta$ and $\kappa$ such that 
\begin{align}\label{eq:paras}
\frac{\vartheta n \gamma}{n-\alpha \gamma} = \left(\vartheta + \frac{\kappa}{\theta }\right)\frac{n}{n-\alpha} = 1,
\end{align}
which guarantees~\eqref{eq:iigoal}.\\

We now consider the other case when $B_{4\varrho}(\xi) \cap \partial\Omega \neq \emptyset$. In this case, there is $z_b \in \partial \Omega$ such that $|z_b - \xi| = \mathrm{dist}(\xi,\partial \Omega) \le 4\varrho$. Let us consider $\tilde{v}$ as the unique solution to the following equation
\begin{equation}\nonumber 
\begin{cases} \mbox{div} \left( \mathcal{A}(x,\nabla \tilde{v})\right) & = \ 0, \quad \ \quad \mbox{ in } \Omega_{12\varrho}(z_b),\\ 
\hspace{1.2cm} \tilde{v} & = \ u, \qquad \mbox{ on } \partial \Omega_{12\varrho}(z_b).\end{cases}
\end{equation}
We remark that $B_{2\varrho}(\xi) \subset B_{6\varrho}(z_b)$, it yields that
\begin{align}\nonumber
L & \le \mathcal{L}^n \left(\left\{ \zeta \in \Omega_{\varrho}(\xi): \ {\mathbf{M}}_{\alpha}(\chi_{B_{6\varrho}(z_b)} |\nabla u|^p)(\zeta)> \varepsilon^{-\vartheta}\lambda \right\} \right)\\ \nonumber
 &\le C  \mathcal{L}^n\left(\left\{\zeta \in \Omega_{\varrho}(\xi): \ {\mathbf{M}}_{\alpha}(\chi_{B_{6\varrho}(z_b)} |\nabla u - \nabla \tilde{v}|^p)(\zeta)> \varepsilon^{-\vartheta}\lambda \right\} \right) \\ \nonumber
& \qquad \qquad + C\mathcal{L}^n\left(\left\{ \zeta \in \Omega_{\varrho}(\xi): \ {\mathbf{M}}_{\alpha}(\chi_{B_{6\varrho}(z_b)} |\nabla \tilde{v}|^p)(\zeta)> \varepsilon^{-\vartheta}\lambda \right\} \right)  \\ \nonumber
&\le C \left( \frac{1}{\varepsilon^{-\vartheta}\lambda} \int_{B_{6\varrho}(z_b)}|\nabla u - \nabla \tilde{v}|^p dx \right)^{\frac{n}{n-\alpha}} + C \left(\frac{1}{\left(\varepsilon^{-\vartheta}\lambda\right)^{\gamma}} \int_{B_{6\varrho}(z_b)} |\nabla \tilde{v}|^{\gamma p} dx\right)^{\frac{n}{n-\alpha \gamma}},
\end{align}
where Lemma~\ref{lem:M_alpha} is applied in the last inequality. It thus gives us
\begin{align} \label{eq:est-101}
L \le C \left( \frac{\varrho^n}{\varepsilon^{-\vartheta}\lambda} \fint_{B_{12\varrho}(z_b)}|\nabla u - \nabla \tilde{v}|^p dx \right)^{\frac{n}{n-\alpha}} + C \left(\frac{\varrho^n}{\left(\varepsilon^{-\vartheta}\lambda\right)^{\gamma}} \fint_{B_{6\varrho}(z_b)} |\nabla \tilde{v}|^{\gamma p} dx\right)^{\frac{n}{n-\alpha \gamma}}.
\end{align}
Thanks to Lemma~\ref{lem:u-w} again, one has
\begin{align*}
\left(\fint_{B_{6\varrho}(z_b)}{|\nabla \tilde{v}|^{\gamma p} dx}\right)^{\frac{1}{\gamma}} \le C \fint_{B_{12\varrho}(z_b)}{|\nabla \tilde{v}|^p dx},
\end{align*}
and
\begin{align} \nonumber
\fint_{B_{12\varrho}(z_b)}{|\nabla u - \nabla \tilde{v}|^pdx}  \le \delta \fint_{B_{12\varrho}(z_b)}{|\nabla u|^p dx} + C {\delta}^{1-\theta } \fint_{B_{12\varrho}(z_b)}{|\mathcal{E}|^pdx},
\end{align}
for all $\delta \in (0,1)$ and $\theta $ given in~\eqref{def:beta}. Similar the first case, using assumptions~\eqref{eq:x2} with notice that $B_{12\varrho}(z_b) \subset B_{18\varrho}(z_1) \cap B_{18\varrho}(z_2)$, we can easily to check that
\begin{align*}
 \fint_{B_{12\varrho}(z_b)}{|\nabla u|^pdx} \le C \varrho^{-\alpha} \lambda \ \mbox{ and } \ \fint_{B_{12\varrho}(z_b)}{|\mathcal{E}|^pdx} \le C \varrho^{-\alpha} \varepsilon^{\kappa} \lambda.
\end{align*}
It follows that
\begin{align}\label{est:B-1}
\fint_{B_{12\varrho}(z_b)}{|\nabla u - \nabla \tilde{v}|^pdx} & \le  C \left(\delta  + \delta^{1-\theta } \varepsilon^{\kappa} \right) \varrho^{-\alpha} \lambda,
\end{align}
and therefore
\begin{align}\nonumber
\fint_{B_{6\varrho}(z_b)} |\nabla \tilde{v}|^{\gamma p} dx & \le C \left(\fint_{B_{12\varrho}(z_b)} |\nabla u|^p + |\nabla u - \nabla \tilde{v}|^pdx \right)^{\gamma} \\ \nonumber
&\le C \left[\left(1+ \delta  + \delta^{1-\theta } \varepsilon^{\kappa} \right) \varrho^{-\alpha} \lambda\right]^{\gamma} \\ \label{est:B-2}
&\le C \left[\left(1 + \delta^{1-\theta } \varepsilon^{\kappa} \right) \varrho^{-\alpha} \lambda\right]^{\gamma}.
\end{align}
Collecting~\eqref{eq:est-101} with~\eqref{est:B-1} and~\eqref{est:B-2}, one gets the same estimate as in~\eqref{est:Lb}. To conclude the proof, it remains to choose the same parameters as in~\eqref{eq:paras}. One concludes~\eqref{eq:iigoal} to finish the proof.
\end{proof} 

\begin{lemma}\label{lem:B2} 
Let $\alpha \in [0,n)$. For all $\varepsilon>0$, one can find some constants $a_{\infty}> 1$, $\delta = \delta(\varepsilon)$ and $a_{\varepsilon} \in (0,1)$ such that if $[\mathbb{A}]_{\tilde{p}}^{r_0} \le \delta$, $z_1 \in \Omega_{\varrho}(\xi)$ and $z_2 \in B_{\varrho}(\xi)$ satisfying
\begin{align}\label{cond:M-B2}
{\mathbf{M}}_{\alpha}(|\nabla u|^p)(z_1) \le \lambda \ \mbox{ and } \
{\mathbf{M}}_{\alpha}(|\mathcal{E}|^p)(z_2) \le a_{\varepsilon} \lambda,
\end{align}
then the following inequality
\begin{align}\label{eq:iigoal-B2}
\mathcal{L}^n \left(\left\{\zeta \in \Omega_{\varrho}(\xi): \ {\mathbf{M}}_{\alpha}^{\varrho} (\chi_{B_{2\varrho}(\xi)} |\nabla u|^p)(\zeta)> a_{\infty} \lambda \right\} \right) \le \varepsilon\mathcal{L}^n\left(B_{\varrho}(\xi)\right),
\end{align}
holds for all $\lambda>0$.
\end{lemma}
\begin{proof}
As in the proof of Lemma~\ref{lem:A3}, we also consider two cases corresponding to $B_{8\varrho}(\xi) \subset \Omega$ and $B_{8\varrho}(\xi) \cap \partial \Omega \neq \emptyset$ respectively. In the first case $B_{8\varrho}(\xi) \subset \Omega$, it notices that $\Omega_{8\varrho}(\xi) = B_{8\varrho}(\xi)$. Let $v$ be the unique solution to 
\begin{equation*}
\begin{cases} \mathrm{div} \left( {\mathbb{A}}(x,\nabla v)\right) & = \ 0, \qquad \, \mbox{ in } B_{8\varrho}(\xi),\\ 
\hspace{1.2cm} v & = \ u-g, \ \mbox{ on } \partial B_{8\varrho}(\xi),\end{cases}
\end{equation*}
and $w$ be the unique solution to the following equation 
\begin{equation*}
\begin{cases} \mathrm{div} \left( \overline{\mathbb{A}}_{B_{4\varrho}(\xi)}(\nabla w)\right) & = \ 0, \ \mbox{ in } B_{4\varrho}(\xi),\\ 
\hspace{1.2cm} w & = \ v, \ \mbox{ on } \partial B_{4\varrho}(\xi).\end{cases}
\end{equation*}
For all $\zeta \in \Omega_{\varrho}(\xi) = B_{\varrho}(\xi)$, thanks to~\eqref{eq:theoIc} in Lemma~\ref{lem:B1} one has 
\begin{align*}
{\mathbf{M}}_{\alpha}^{\varrho}(\chi_{B_{2\varrho}(\xi)} |\nabla w|^p)(\zeta) & \le C \varrho^{\alpha}\|\nabla w\|^p_{L^{\infty}(B_{2\varrho}(\xi))} \\
& \le C \left(\varrho^{\alpha} \fint_{B_{8\varrho}(\xi)} |\nabla u|^p  dx +    \varrho^{\alpha}\fint_{B_{8\varrho}(\xi)}  |\mathcal{E}|^p dx \right) \\
& \le C \left( {\mathbf{M}}_{\alpha}(|\nabla u|^p)(z_1) + {\mathbf{M}}_{\alpha}(|\mathcal{E}|^p)(z_2) \right),
\end{align*}
which deduces from the assumption~\eqref{cond:M-B2} that
\begin{align*}
{\mathbf{M}}_{\alpha}^{\varrho}(\chi_{B_{2\varrho}(\xi)} |\nabla w|^p)(\zeta) & \le C \left(1+a_{\varepsilon}\right)\lambda \le C^* \lambda, \ \mbox{ for all } a_{\varepsilon} \in (0,1).
\end{align*}
It yields that if we choose $a_{\infty} \ge C^*$ then $$ \mathcal{L}^n \left(\left\{\zeta \in B_{\varrho}(\xi): \ {\mathbf{M}}_{\alpha}^{\varrho} (\chi_{B_{2\varrho}(\xi)} |\nabla w|^p)(\zeta)> a_{\infty} \lambda \right\} \right) = 0,$$ 
which guarantees that
\begin{align}\nonumber
\mathcal{L}^n & \left(\left\{\zeta \in \Omega_{\varrho}(\xi): \ {\mathbf{M}}_{\alpha}^{\varrho} (\chi_{B_{2\varrho}(\xi)} |\nabla u|^p)(\zeta)> a_{\infty} \lambda \right\} \right) \\ \nonumber
& \hspace{2cm} \le C \mathcal{L}^n \left(\left\{\zeta \in B_{\varrho}(\xi): \ {\mathbf{M}}_{\alpha}^{\varrho} (\chi_{B_{2\varrho}(\xi)} |\nabla u - \nabla w|^p)(\zeta)> a_{\infty} \lambda \right\} \right) \\ \label{est:B2-1}
& \hspace{2cm} \le  C \left(\frac{1}{a_{\infty} \lambda}\int_{B_{2\varrho}(\xi)}|\nabla u - \nabla w|^p dx\right)^{\frac{n}{n-\alpha}}.
\end{align}
Here the last inequality in~\eqref{est:B2-1} holds due to Lemma~\ref{lem:M_alpha} with $s = 1$. The right-hand side of~\eqref{est:B2-1} can be estimated by applying inequality~\eqref{eq:theoIc} in Lemma~\ref{lem:B1} and assumption~\eqref{cond:M-B2}, as follows
\begin{align}\nonumber
\int_{B_{2\varrho}(\xi)}|\nabla u - \nabla w|^p dx & \le C \varrho^n \fint_{B_{4\varrho}(\xi)} |\nabla u - \nabla w|^pdx \\ \nonumber
& \le C \varrho^{n - \alpha} \left( \left([\mathbb{A}]_{\tilde{p}}^{4\varrho}\right)^p \varrho^{\alpha}\fint_{B_{8\varrho}(\xi)} |\nabla u|^p dx \right. \\ \nonumber 
& \hspace{3cm} \left. +  \left([\mathbb{A}]_{\tilde{p}}^{4\varrho}\right)^{p(1-\theta )}  \varrho^{\alpha}\fint_{B_{8\varrho}(\xi)} |\mathcal{E}|^p dx \right) \\ \nonumber
& \le C \varrho^{n - \alpha}\left( \delta^p {\mathbf{M}}_{\alpha}(|\nabla u|^p)(z_1) + \delta^{p(1-\theta )} {\mathbf{M}}_{\alpha}(|\mathcal{E}|^p)(z_2) \right) \\ 
\nonumber
& \le C \varrho^{n - \alpha}\left( \delta^p  + \delta^{p(1-\theta )} a_{\varepsilon} \right) \lambda \\ \label{est:B2-2}
& \le C \varrho^{n - \alpha} \delta^p  \lambda,
\end{align} 
where $a_{\varepsilon} \in (0,\min\{1,\delta^{p\theta }\})$. Combining~\eqref{est:B2-1} and~\eqref{est:B2-2}, we obtain that
\begin{align*}
\mathcal{L}^n & \left(\left\{\zeta \in \Omega_{\varrho}(\xi): \ {\mathbf{M}}_{\alpha}^{\varrho} (\chi_{B_{2\varrho}(\xi)} |\nabla u|^p)(\zeta)> a_{\infty} \lambda \right\} \right) \le C \varrho^n \left(\frac{\delta^p}{a_{\infty}}\right)^{\frac{n}{n-\alpha}},
\end{align*}
which implies to~\eqref{eq:iigoal-B2} by taking $\delta$ small enough such that $C\left(\frac{\delta^p}{a_{\infty}}\right)^{\frac{n}{n-\alpha}} < \varepsilon$. We can prove~\eqref{eq:iigoal-B2} for the second case $B_{8\varrho}(\xi) \cap \partial \Omega \neq \emptyset$ by the same technique as in the proof of Lemma~\ref{lem:A3}. More precisely, we first take $z_b \in \partial \Omega$ such that $|z_b - \xi| = \mathrm{dist}(\xi,\partial \Omega) \le 8\varrho$. Then we consider the similar problem in $\Omega_{18\varrho}(z_b)$ as the proof of the first case. 
\end{proof}

\section{Proofs of main Theorems}\label{sec:proofs}
In this section, we give detail proofs of main Theorems by applying the covering Lemma~\ref{lem:mainlem} which is discussed in the previous section.
 
\begin{proof}[Proof of Theorem~\ref{theo:good-lambda}]
In this proof, we apply Lemma~\ref{lem:mainlem} for two measurable subsets $\mathcal{V}_{\varepsilon, \lambda}$, $\mathcal{W}_{\lambda}$ of $\Omega$ given below. For every $\varepsilon>0$ and $\lambda>0$, let us define
\begin{align*}
 \mathcal{V}_{\varepsilon, \lambda}  & := \left\{\zeta \in \Omega: \ {\mathbf{M}}_{\alpha}(|\nabla u|^p)(\zeta)>\varepsilon^{\frac{\alpha}{n} - \frac{1}{\gamma}}\lambda, \ {\mathbf{M}}_{\alpha}(|\mathcal{E}|^p)(\zeta) \le \varepsilon^{\theta  \left(1 - \frac{1}{\gamma}\right)}\lambda \right\}, \\ 
& \hspace{2cm} \mbox{and } \ \mathcal{W}_{\lambda} := \left\{\zeta \in \Omega: \ {\mathbf{M}}_{\alpha}(|\nabla u|^p)(\zeta)> \lambda \right\},
\end{align*}
where $\gamma>1$ and $\theta  \ge 2$ are two constants defined as in Lemma~\ref{lem:u-w}. To prove the good-$\lambda$ inequality~\eqref{eq:good-lambda} in Theorem~\ref{theo:good-lambda}, we need to show that for every $0< R < r_0/12$, there holds
\begin{enumerate}
\item[i)] $\mathcal{L}^n\left(\mathcal{V}_{\varepsilon, \lambda}\right) \le  \varepsilon \mathcal{L}^n \left(\mathcal{B}_R(0)\right)$;
\item[ii)] for any $\xi \in \Omega$, $\varrho \in (0,R]$, if $\mathcal{L}^n\left(\mathcal{V}_{\varepsilon, \lambda} \cap B_{\varrho}(\xi)\right) > \varepsilon \mathcal{L}^n\left(B_{\varrho}(\xi)\right)$ then $\Omega_{\varrho}(\xi) \subset \mathcal{W}_{\lambda}$;
\end{enumerate}
for all $\lambda>0$ and $\varepsilon$ small enough. We first emphasize that it is nothing to do when $\mathcal{V}_{\varepsilon, \lambda}$ is empty. So we may assume that $\mathcal{V}_{\varepsilon, \lambda}$ is not empty which implies that there exists $\xi_1 \in \Omega$ such that ${\mathbf{M}}_{\alpha}(|\mathcal{E}|^p)(\xi_1) \le \varepsilon^{\theta  \left(1 - \frac{1}{\gamma}\right)}\lambda$. Thanks to Lemma~\ref{lem:A1}, one has
\begin{align*}
\mathcal{L}^n\left(\mathcal{V}_{\varepsilon, \lambda}\right) \le C  \varepsilon^{1 + \frac{n(\gamma-1)(\theta -1)}{\gamma(n-\alpha)}} \mathcal{L}^n\left(\mathcal{B}_R(0)\right),
\end{align*}
which guarantees $i)$ for $\varepsilon$ small enough. The proof of $ii)$ can be obtained by contradiction. Let us assume that we can find $\xi_2 \in \Omega_{\varrho}(\xi) \cap \mathcal{W}^c_{\lambda}$ and $\xi_3 \in \mathcal{V}_{\varepsilon, \lambda} \cap B_{\varrho}(\xi)$, it leads to
\begin{align*}
{\mathbf{M}}_{\alpha}(|\nabla u|^p)(\xi_2) \le \lambda \ \mbox{ and } \
{\mathbf{M}}_{\alpha}(|\mathcal{E}|^p)(\xi_3) \le \varepsilon^{\theta  \left(1 - \frac{1}{\gamma}\right)}\lambda.
\end{align*}
Thanks to Lemma~\ref{lem:A2} and Lemma~\ref{lem:A3}, for $\varepsilon \in \left(0,3^{-\frac{n^2\gamma}{n-\alpha \gamma}}\right)$ there holds
\begin{align*}
\mathcal{L}^n\left(\mathcal{V}_{\varepsilon, \lambda} \cap B_{\varrho}(\xi)\right) & \le \mathcal{L}^n \left(\left\{\zeta \in \Omega_{\varrho}(\xi): \ {\mathbf{M}}_{\alpha}(\chi_{B_{2\varrho}(\xi)} |\nabla u|^p)(\zeta) > \varepsilon^{\frac{\alpha}{n}-\frac{1}{\gamma}} \lambda \right\} \right)  \le \varepsilon \mathcal{L}^n \left(B_{\varrho}(\xi)\right),
\end{align*}
which finishes the proof.
\end{proof}

\begin{proof}[Proof of Theorem~\ref{theo:main}]
By changing of variable $\lambda$ to $\varepsilon^{\frac{\alpha}{n}-\frac{1}{\gamma}}\lambda$ in the integral of definition of norm in Lorentz space $L^{q,s}(\Omega)$, there holds
\begin{align*}
\|\mathbf{M}_{\alpha}(|\nabla u|^p)\|^s_{L^{q,s}(\Omega)} & = q \int_0^\infty{\lambda^s \mathcal{L}^n\left(\left\{\zeta \in \Omega: \ \mathbf{M}_{\alpha}(|\nabla u|^p)(\zeta)>\lambda\right\} \right)^{\frac{s}{q}}\frac{d\lambda}{\lambda}} \\
& = \varepsilon^{\frac{s\alpha}{n}-\frac{s}{\gamma}}q\int_0^\infty{\lambda^s\mathcal{L}^n\left(\left\{\zeta \in \Omega: \ \mathbf{M}_{\alpha}(|\nabla u|^p)(\zeta)>\varepsilon^{\frac{\alpha}{n}-\frac{1}{\gamma}}\lambda\right\} \right)^{\frac{s}{q}}\frac{d\lambda}{\lambda}}.
\end{align*}
Due to Theorem~\ref{theo:good-lambda}, one has
\begin{align*}
& \mathcal{L}^n\left(\left\{\zeta \in \Omega: \ \mathbf{M}_{\alpha}(|\nabla u|^p)(\zeta)>\varepsilon^{\frac{\alpha}{n}-\frac{1}{\gamma}}\lambda\right\}\right) \le C\varepsilon \mathcal{L}^n\left(\left\{\zeta \in \Omega: \ \mathbf{M}_{\alpha}(|\nabla u|^p)(\zeta)>\lambda\right\}\right)\\ & \hspace{2cm} + \mathcal{L}^n\left(\left\{\zeta \in \Omega: \ \mathbf{M}_{\alpha}(|\mathcal{E}|^p)(\zeta)>\varepsilon^{\theta \left(1-\frac{1}{\gamma}\right)}\lambda\right\} \right),
\end{align*}
for all $\varepsilon \in (0,\varepsilon_0)$ for some $\varepsilon_0>0$ small enough. Therefore, we get that
\begin{align}\nonumber
& \|\mathbf{M}_{\alpha}(|\nabla u|^p)\|^s_{L^{q,s}(\Omega)} \le C\varepsilon^{\frac{s\alpha}{n}-\frac{s}{\gamma}+\frac{s}{q}}q\int_0^\infty{\lambda^s\mathcal{L}^n\left(\left\{\zeta \in \Omega: \ \mathbf{M}_{\alpha}(|\nabla u|^p)(\zeta)>\lambda\right\} \right)^{\frac{s}{q}}\frac{d\lambda}{\lambda}}\\ \label{est:5a}
& \hspace{1cm} + C\varepsilon^{\frac{s\alpha}{n}-\frac{s}{\gamma}}q\int_0^\infty{\lambda^s\mathcal{L}^n\left(\left\{\zeta \in \Omega: \  \mathbf{M}_{\alpha}(|\mathcal{E}|^p)(\zeta)>\varepsilon^{\theta \left(1-\frac{1}{\gamma}\right)}\lambda\right\} \right)^{\frac{s}{q}}\frac{d\lambda}{\lambda}}.
\end{align}
By changing of variable again in the last term on right-hand side of~\eqref{est:5a}, it implies that
\begin{align}\nonumber
\|\mathbf{M}_{\alpha}(|\nabla u|^p)\|^s_{L^{q,s}(\Omega)} &\le C\varepsilon^{\frac{s\alpha}{n}-\frac{s}{\gamma}+\frac{s}{q}}\|\mathbf{M}_{\alpha}\left(|\nabla u|^p \right)\|^s_{L^{q,s}(\Omega)}\\  \label{est:5b}
& \qquad +C\varepsilon^{\frac{s\alpha}{n}-\frac{s}{\gamma}-s\theta \left(1-\frac{1}{\gamma}\right)}\|\mathbf{M}_{\alpha}(|\mathcal{E}|^p)\|^s_{L^{q,s}(\Omega)}.
\end{align}
For every $0<s<\infty$ and $0<q<\frac{n\gamma}{n-\alpha\gamma}$, we may choose $\varepsilon$ in~\eqref{est:5b} such that:
\begin{align*}
C\varepsilon^{s\left(\frac{\alpha}{n}-\frac{1}{\gamma}+ \frac{1}{q}\right)} \le \frac{1}{2},
\end{align*}
to obtain~\eqref{eq:regularityMalpha}. The proof can be obtained by the similar way for the case $s = \infty$. 
\end{proof}

\begin{proof}[Proof of Theorem~\ref{theo:good-lambda-Rf} ]
It is similar to the proof of Theorem~\ref{theo:good-lambda}, the good-$\lambda$ inequality~\eqref{eq:good-lambda-Rf} can be proved by applying Lemma~\ref{lem:mainlem} for two following subsets of $\Omega$ as
\begin{align*}
 \mathcal{V}_{\varepsilon, \lambda}  & := \left\{\zeta \in \Omega: \ {\mathbf{M}}_{\alpha}(|\nabla u|^p)(\zeta)> a_{\infty}\lambda, \ {\mathbf{M}}_{\alpha}(|\mathcal{E}|^p)(\zeta) \le a_{\varepsilon} \lambda \right\}, \\ 
& \hspace{1cm} \mbox{and } \ \mathcal{W}_{\lambda} := \left\{\zeta \in \Omega: \ {\mathbf{M}}_{\alpha}(|\nabla u|^p)(\zeta)> \lambda \right\},
\end{align*}
where $a_{\infty}$ and $a_{\varepsilon}$ will be determined later. Without loss of generality, we may assume that $\mathcal{V}_{\varepsilon, \lambda}$ is not empty which implies that there exists $\xi_1 \in \Omega$ such that ${\mathbf{M}}_{\alpha}(|\mathcal{E}|^p)(\xi_1) \le  a_{\varepsilon} \lambda$. Thanks to Lemma~\ref{lem:A1}, one has
\begin{align*}
\mathcal{L}^n\left(\mathcal{V}_{\varepsilon, \lambda}\right) \le  C^{**}  \left(\frac{ a_{\varepsilon}}{a_{\infty}}\right)^{\frac{n}{n-\alpha}} \mathcal{L}^n\left(\mathcal{B}_{r_0}(0)\right) \le \varepsilon \mathcal{L}^n\left(\mathcal{B}_{r_0}(0)\right),
\end{align*}
for all $a_{\varepsilon}$ small enough such that $C^{**} \left({ a_{\varepsilon}}/{a_{\infty}}\right)^{\frac{n}{n-\alpha}} \le \varepsilon$. The second assumption of Lemma~\ref{lem:mainlem} will be showed by contradiction. The existence of two points $\xi_2 \in \Omega_{\varrho}(\xi) \cap \mathcal{W}^c_{\lambda}$ and $\xi_3 \in \mathcal{V}_{\varepsilon, \lambda} \cap B_{\varrho}(\xi)$ leads to 
$${\mathbf{M}}_{\alpha}(|\nabla u|^p)(\xi_2) \le \lambda \ \mbox{ and } \ {\mathbf{M}}_{\alpha}(|\mathcal{E}|^p)(\xi_3) \le a_{\varepsilon} \lambda.$$ 
Lemma~\ref{lem:A2} and Lemma~\ref{lem:B2} ensure that there exist $\delta = \delta(n,\alpha,\varepsilon) \in (0,1)$, $a_{\infty} = a_{\infty}(n,p,\alpha)>1$ and $a_{\varepsilon} = a_{\varepsilon}(n,\alpha,\varepsilon) \in (0,1)$ such that if $\Omega$ is a $(\delta,r_0)$-Reifenberg flat domain satisfying $[\mathbb{A}]_{\tilde{p}}^{r_0} \le \delta$ then
\begin{align*}
\mathcal{L}^n\left(\mathcal{V}_{\varepsilon, \lambda} \cap B_{\varrho}(\xi)\right) & \le \mathcal{L}^n \left(\left\{\zeta \in \Omega_{\varrho}(\xi): \ {\mathbf{M}}_{\alpha}^{\varrho}(\chi_{B_{2\varrho}(\xi)} |\nabla u|^p)(\zeta) > a_{\infty} \lambda \right\} \right) \le \varepsilon \mathcal{L}^n \left(B_{\varrho}(\xi)\right).
\end{align*}
The proof is completed by applying Lemma~\ref{lem:mainlem}.
\end{proof}

\begin{proof}[Proof of Theorem~\ref{theo:main-Rf}]
Theorem~\ref{theo:good-lambda-Rf} show that one can find $a_{\infty}$ and $a_{\varepsilon}$ such that the following inequality
\begin{align*}
& \mathcal{L}^n\left(\{\zeta \in \Omega: \ \mathbf{M}_{\alpha}(|\nabla u|^p)(\zeta)> a_{\infty} \lambda\}\right) \le C\varepsilon \mathcal{L}^n\left(\{\zeta \in \Omega: \ \mathbf{M}_{\alpha}(|\nabla u|^p)(\zeta)>\lambda\}\right)\\ & \hspace{4cm} + \mathcal{L}^n\left(\{\zeta \in \Omega: \ \mathbf{M}_{\alpha}(|\mathcal{E}|^p)(\zeta)> a_{\varepsilon}\lambda\} \right),
\end{align*}
holds for all $\varepsilon >0$. Therefore, for every $0<q<\infty$ and $0<s<\infty$, by  changing of variables we get that
\begin{align}\nonumber
\|\mathbf{M}_{\alpha}(|\nabla u|^p)\|^s_{L^{q,s}(\Omega)} & = a_{\infty}^s q\int_0^\infty{\lambda^s\mathcal{L}^n\left(\{\zeta \in \Omega: \ \mathbf{M}_{\alpha}(|\nabla u|^p)(\zeta)> a_{\infty} \lambda\} \right)^{\frac{s}{q}}\frac{d\lambda}{\lambda}}\\ \nonumber
& \le C a_{\infty}^s \varepsilon^{\frac{s}{q}} q\int_0^\infty{\lambda^s\mathcal{L}^n\left(\{\zeta \in \Omega: \ \mathbf{M}_{\alpha}(|\nabla u|^p)(\zeta)>\lambda\} \right)^{\frac{s}{q}}\frac{d\lambda}{\lambda}}\\ \nonumber
& \hspace{1cm} + C a_{\infty}^s q\int_0^\infty{\lambda^s\mathcal{L}^n\left(\{\zeta \in \Omega: \  \mathbf{M}_{\alpha}(|\mathcal{E}|^p)(\zeta)> a_{\varepsilon} \lambda\} \right)^{\frac{s}{q}}\frac{d\lambda}{\lambda}} \\ \nonumber
&\le C a_{\infty}^s \varepsilon^{\frac{s}{q}} \|\mathbf{M}_{\alpha}\left(|\nabla u|^p \right)\|^s_{L^{q,s}(\Omega)}  +C (a_{\infty}/a_{\varepsilon})^{s} \|\mathbf{M}_{\alpha}(|\mathcal{E}|^p)\|^s_{L^{q,s}(\Omega)}.
\end{align}
Let us fix $\varepsilon$ in the last inequality such that $C a_{\infty}^s \varepsilon^{\frac{s}{q}} \le \frac{1}{2}$ to have~\eqref{eq:regularityMalpha}. It is similar for the case $s = \infty$. 
\end{proof}

\end{document}